\documentclass[a4paper,10pt]{article}
\usepackage{amscd}
\usepackage{amsmath}
\usepackage{bbding}
\usepackage{amsfonts}
\usepackage{cite}
\usepackage{mathrsfs}
\usepackage{graphicx}
\usepackage{subfigure}
\usepackage{amssymb}
\usepackage{graphicx,latexsym,amsmath,amssymb,amsthm,enumerate}
9.4in\textwidth 6.1in \hoffset -2.0cm \DeclareFontEncoding{OT2}{}{}
\DeclareTextSymbol{\cyrsftsn}{OT2}{126}
\DeclareTextSymbol{\textnumero}{OT2}{125}

\theoremstyle{definition}
\newtheorem{theorem}{Theorem}[section]
\newtheorem{lemma}{Lemma}[section]

\newtheorem{definition}{Definition}[section]
\newtheorem{remark}{Remark}[section]

\begin{document}
\title{{\LARGE\bf{ Fractional Wishart Processes and $\varepsilon$-Fractional Wishart Processes with Applications}\thanks{This work was supported by  the National Natural Science Foundation of China (11471230, 11671282).}}}
\author{Jia Yue and Nan-jing Huang \footnote{Corresponding author.  E-mail addresses: nanjinghuang@hotmail.com, njhuang@scu.edu.cn} \\
{\small\it Department of Mathematics, Sichuan University, Chengdu,
Sichuan 610064, P.R. China}}
\date{ }
\maketitle
\begin{flushleft}
\hrulefill\\
\end{flushleft}
 {\bf Abstract}.
In this paper, we introduce two new matrix stochastic processes: fractional Wishart processes and  $\varepsilon$-fractional Wishart processes with integer indices which are based on the fractional Brownian motions and then extend $\varepsilon$-fractional Wishart processes to the case with non-integer indices. Both processes include classic Wishart processes (if the Hurst index $H$ equals $\frac{1}{2}$) and present serial correlation of stochastic processes. Applying $\varepsilon$-fractional Wishart processes to  financial volatility theory, the financial models account for the stochastic volatilities of the assets and for the stochastic correlations not only between the underlying assets' returns but also between their volatilities and for stochastic serial correlation of the relevant assets.
\begin{flushleft} \hrulefill \end{flushleft}

\section{Introduction}\label{section1} \noindent
\setcounter{equation}{0}
Since Black and Scholes' significant work (\cite{Black}), more and more stochastic processes are widely used to capture diverse phenomena in financial markets, such as Brownian motions, fractional Brownian motions, L\'{e}vy processes, Wishart processes and so on.

Heston's model (\cite{Heston}) adopts Brownian motions to describe the stochastic volatility of the stock, as empirical evidences (\cite{Anderson, Dupire}) have shown that the classic Black-Scholes assumption of lognormal stock diffusion with constant volatility is not consistent with the market price (such phenomenon is often referred to as the volatility skew or smile).

As the backbone of multivariate statistical analysis, random matrices have found their applications in many fields, such as physics, economics, psychology and so on. Bru (\cite{Bru}) develops Wishart processes in mathematic which are dynamic random matrices and turn out to be a better way to capture stochastic volatility and correlation structure of the relevant random vectors. In recent years, there has been tremendous growth of multi-asset financial contracts (outperformance options, for example), which exhibit sensitivity to both the volatilities and the correlations of the underlying assets. The authors of \cite{Avellaneda} and \cite{Cont} show that the correlations between financial assets evolve stochastically and are far from remaining static through time. Furthermore, in \cite{Loretan} and \cite{Solnik}, there are evidences which present that the higher the market volatility is, the higher the correlations between financial assets tend to be. In order to include those phenomena in financial markets,  Wishart Affine Stochastic Correlation models (\cite{Fonseca,Romo15,Nicole}) introduce Wishart processes to account for the stochastic volatilities of the assets and for the stochastic correlations not only between the underlying assets' returns but also between their volatilities.

There also exists early evidence (\cite{Lo}) which shows the processes of observable market values seem to exhibit serial correlation (this means the increments of the process depend on the information of the past). So fractional Brownian motions are proposed for mapping this kind of behavior. Fractional Brownian motions not only capture serial correlation of stochastic processes, but also keep a good analytical treatability for still being Gaussian, which leads that they become an interesting and important candidate for financial models. Furthermore, generalized results of fractional Brownian motions often include the corresponding well-known results of classic Brownian motions, as fractional Brownian motions are just classic Brownian motions when the Hurst index equals $\frac{1}{2}$.
Mandelbrot and van Ness (\cite{Mandelbrot}) suggest fractional Brownian motions as an alternative model for assets' dynamics, which allow for dependence between returns over time. Since then, there is an ongoing dispute on the usage of fractional Brownian motions in financial theories (\cite{Willinger, Sottinen, Rogers}).

In our paper, we shall introduce two new matrix stochastic processes: the fractional Wishart process and  $\varepsilon$-fractional Wishart process. First of all, the fractional Wishart process (which is based on the fractional Brownian motion) is the generalization of the Wishart process (which is based on the classic Brownian motion) such that the former degenerates to the latter when the Hurst index equals $\frac{1}{2}$. The fractional Wishart process can present serial correlation of stochastic processes while the Wishart process is a Markov process (see Definition 2 in \cite{Gourieroux06}) whose increments are independent of the past such that the Wishart process can be thought of as `memoryless'. The $\varepsilon$-fractional Wishart process is the approximation of the fractional Wishart process as the fractional Wishart process does not keep a good analytical treatability. The difference between the Wishart process and $\varepsilon$-fractional Wishart process is that the former process is governed by a related stochastic differential equation (SDE) while the latter process is governed by a related stochastic partial differential equation (SPDE). Of course, the $\varepsilon$-fractional Wishart process becomes the Wishart process when the Hurst index equals $\frac{1}{2}$. In financial theory, if we apply the fractional Wishart process or  $\varepsilon$-fractional Wishart process  to the volatility of assets, then the model shall account for the stochastic volatilities of the assets and for the stochastic correlations not only between the underlying assets' returns but also between their volatilities and for stochastic serial correlation of the relevant assets.

The rest of the paper is organized as follows. Section \ref{section2} sketches the main assumptions and results needed in this paper. Then we define the fractional Wishart process with an integer index in Section \ref{section3}. In Section \ref{section4}, we define $\varepsilon$-fractional Wishart process with an integer index and then extend $\varepsilon$-fractional Wishart process to the case with a non-integer index by its related SPDE. In Section \ref{section5}, a generalization of the $\varepsilon$-fractional Wishart process which includes two more parameters is discussed.  In Section \ref{section6}, we apply the $\varepsilon$-fractional Wishart process to  financial volatility model.  Finally, some conclusions and future work are included in Section \ref{section7}.

\section{Preliminaries}\label{section2}\noindent
\setcounter{equation}{0}
In this section, we shall sketch some basic concepts related to matrix variate distributions and fractional Brownian motions and so on (one can refer to \cite{Gupta, Kunita, Biagini} for details).

Let $(\mathbf{\Omega},\mathcal{G},(\mathcal{G}_t)_{t\geq 0}, \mathbb{P}_0)$ be a filtered probability space satisfying the usual conditions. The stochastic processes are considered in such probability space if we do not give the probability space.

For any positive integers $n$ and $p$, let $\mathcal{M}_{n,p}(\mathbb{R})$ (resp. $\mathcal{S}_{p}(\mathbb{R})$, $\mathcal{S}_{p}^+(\mathbb{R})$ and $\overline{\mathcal{S}_p^+}(\mathbb{R})$ ) denote the sets of all real-valued $n\times p$ matrices (resp. $p\times p$ symmetric matrices,  $p\times p$ symmetric positive definite matrices and $p\times p$ symmetric positive semidefinite matrices). For any real Banach spaces $\mathbb{X}$ and $\mathbb{Y}$, let $\mathcal{L}(\mathbb{X},\mathbb{Y})$ denote the space of bounded linear operators from $\mathbb{X}$ to $\mathbb{Y}$. In $\mathcal{M}_{n,p}(\mathbb{R})$£¬ the norm  of any matrix $A=(A_{ij})_{1\leq i\leq n,1\leq j\leq p}$, denoted by $|A|_{n\times p}$ or just $|A|$ is defined as the Frobenius norm (see, for example, \cite{Dennis}):
\[|A| = \sqrt{\sum\limits_{i = 1}^n {\sum\limits_{j = 1}^p {|{A_{ij}}{|^2}} } }.\]
 For normed spaces $\mathbb{X}_1$ and $\mathbb{X}_2$, we will define the norm on $\mathbb{X}_1\times \mathbb{X}_2$ as $|\cdot|_{\mathbb{X}_1}+|\cdot|_{\mathbb{X}_2}$ (or $|\cdot|+|\cdot|$). For any map $f(x_1,\cdots,x_n)$, let $\partial_{x_i}^kf$ denote the $k$-th partial derivative of $f$ and $\partial_{x_1}^kf$ be denoted by $\partial^kf$ when $n=1$. Moreover, we always use $X'$ to denote the transpose of a matrix $X$.

For a definition of a random matrix, as well as its probability density function (p.d.f.), the moment generating function (m.g.f) and so on, one can refer to \cite{Gupta}. For any random matrix $X\in \mathcal{M}_{n,p}( \mathbb{R})$, we mean that the random matrix $X$ takes its values in the set $\mathcal{M}_{n,p}( \mathbb{R})$.

Let $A\in \mathcal{M}_{m,n}(\mathbb{R})$ and $B\in \mathcal{M}_{p,q}(\mathbb{R})$. Then the Kronecker product (also called the direct product, see, for example, \cite{Gupta}) of $A$ and $B$, denoted by $A\otimes B$, is defined by
\[A\otimes B = ({A_{ij}}B) = \left( {\begin{array}{*{20}{c}}
{{A_{11}}B}&{{A_{12}}B}& \cdots &{{A_{1n}}B}\\
{{A_{21}}B}&{{A_{21}}B}& \cdots &{{A_{11}}B}\\
 \vdots & \vdots & \ddots & \vdots \\
{{A_{m1}}B}&{{A_{m1}}B}& \cdots &{{A_{mn}}B}
\end{array}} \right).\]
For a matrix $X\in \mathcal{M}_{n,p}( \mathbb{R})$, let ${\mathop{\rm vec}}(X)$ be the following $np\times 1$ vector,
\[{\mathop{\rm vec}}(X) = \left( {\begin{array}{*{20}{c}}
{{x_1}}\\
 \vdots \\
{{x_p}}
\end{array}} \right) = \left( {\begin{array}{*{20}{c}}
{X{e_1}}\\
 \vdots \\
{X{e_p}}
\end{array}} \right),\]
where $x_i$ ( $i=1,\cdots,p$ ) is the $i^{th}$ column of $X$ and  $e_1=(1,0,\cdots,0)',\cdots,e_p=(0,0,\cdots,1)'$.
If $X\in \mathcal{M}_{n,p}( \mathbb{R})$ and $Y\in \mathcal{M}_{r,s}( \mathbb{R})$ are two random matrices, then the $np\times rs$ covariance matrix is defined by $ {\mathop{\rm cov}} (X,Y) = {\mathop{\rm cov}} (\mathrm{vec}(X'),\mathrm{vec}(Y'))$ and specially, ${\mathop{\rm cov}} (X) =  {\mathop{\rm cov}} (\mathrm{vec}(X'))$.

\begin{definition}\label{wd}
A random matrix $S\in \mathcal{S}_{p}^+( \mathbb{R})$ is said to have a noncentral Wishart distribution with parameters $p,n(n\geq p)$, $\Sigma\in \mathcal{S}_{p}^+( \mathbb{R})$ and $\Theta \in \mathcal{M}_{p,p}(\mathbb{R})$, written as $S \sim {W_p}(n,\Sigma,\Theta )$, if its p.d.f is given by
\[{\left\{ {2^{\frac{1}{2}np}}{\Gamma _p}\left(\frac{1}{2}n\right)\det {(\Sigma )^{\frac{1}{2}n}}\right\} ^{ - 1}}\det {(S)^{\frac{1}{2}(n - p - 1)}}{\mathop{\rm etr}}\left( - \frac{1}{2}\Theta  - \frac{1}{2}{\Sigma ^{ - 1}}S\right)\ _0{F_1}\left(\frac{1}{2}n;\frac{1}{4}\Theta {\Sigma ^{ - 1}}S\right),\]
where $\det(\cdot)$ denotes the determinant of a square matrix, ${\mathop{\rm Tr}}(\cdot)$ denotes the trace of a square matrix, ${\mathop{\rm etr}}(\cdot)$ is short for $\exp\{{\mathop{\rm Tr}}(\cdot)\}$, $\Gamma _p(\cdot)$ is the multivariate gamma function and $_0F_1$ is the hypergeometric function (\cite{Gupta}).
\end{definition}

\begin{definition}\label{d2.3}
Let $H\in (0,1)$ be a constant. A fractional Brownian motion $(B_t^H)_{t\geq 0}$ of Hurst index $H$ is a continuous Gaussian process with covariance function
\[\mathrm{cov}[B_t^H,B_s^H] = \frac{1}{2}({t^{2H}} + {s^{2H}} - |t - s{|^{2H}})£¬\]
with initial state $B_0^H=C$ almost surely (a.s.).
\end{definition}

Note that a fractional Brownian motion is just a Brownian motion if $H=\frac {1}{2}$.

\begin{remark}
According to \cite{Rostek}, in the case $H<\frac{1}{2}$, the fractional Brownian motion is also called the antipersistent fractional Brownian motion which has intermediate memory, whereas in the case $H>\frac{1}{2}$, the fractional Brownian motion is also called the persistent fractional Brownian motion which has long memory. Hence, as well-known to us, in the case $H\neq\frac{1}{2}$, the fractional Brownian motion increment depends on its historical increments, while the classic Brownian motion has independent increments which are independent of the past.
\end{remark}

For other concepts, one can refer to \cite{Barndorff-Nielsen} for matrix stochastic process and to \cite{Kunita} for the random field. Here we shall give some important results in the form of matrix.

\begin{definition}
For two semimartingales $A\in \mathcal{M}_{d,m}( \mathbb{R})$ and $B\in \mathcal{M}_{m,n}( \mathbb{R})$, the matrix quadratic covariation process is defined by
\[{\left\langle {A,B} \right\rangle _t} = {\left(\sum\limits_{k = 1}^m {{{\left\langle {{A_{ik}},{B_{kj}}} \right\rangle }_t}} \right)_{ij}} \in \mathcal{M}_{d,n}( \mathbb{R}).\]
\end{definition}

\begin{remark}
It is known that $\left\langle {\cdot,\cdot} \right\rangle$ plays the same role for matrix multiplication of matrix-valued semimartingales, as the quadratic covariation process does for multiplication of one-dimensional semimartingales (\cite{Barndorff-Nielsen}). We will use the following symbol
\[{\left\langle A \right\rangle _t} = {\left\langle {A,A'} \right\rangle _t}\in \mathcal{M}_{d,d}( \mathbb{R}).\]
For example, let $B_t=(B_1(t),\cdots,B_n(t))'$ be an $n$-dimensional standard Brownian motion in $\mathcal{M}_{n,1}( \mathbb{R})$.  Then $\left\langle B \right\rangle _t=tI_n$.
\end{remark}

Throughout, we shall use $C(q)$  to denote a positive constant depending only on $q$ but its value may change from line to line. Moreover, all integrals with respect to Brownian motions or Brownian matrices are in the sense of It\^{o} satisfying the conditions such that It\^{o} integrals exist (\cite{Oksendal}).

\begin{lemma}\label{MMI}
(Martingale Moment Inequality) Let $M$ be a continuous local martingale in $\mathcal{M}_{p,p}( \mathbb{R})$, along with its quadratic variation process $\langle M\rangle$. For every $q>0$, there exists a positive constant $C(q)$ such that, for any stopping time $T$,
\begin{equation}\label{eqMMI_1}
E|{M_T}|^q \le C(q)E| \langle M{ \rangle _T}|^{\frac{q}{2}}.
\end{equation}
\end{lemma}

\begin{proof}
For any $i,j=1,2,\cdots,p$, by Burkholder-Davis-Gundy inequalities (see Theorem 3.28 in \cite{Karatzas}), we have
\begin{equation}\label{eqMMI_2}
E|{M_{ij,T}}{|^q} \le {C}(q)E| \langle {M_{ij}}{ \rangle _T}{|^{\frac{q}{2}}} \le {C}(q)E| \langle M{ \rangle _T}{|^{\frac{q}{2}}}.
\end{equation}

For any $a_i\in \mathbb{R}$ and integer $n$, it is easy to check that
\[\sum\limits_{i = 1}^n {|{a_i}{|^q}}  \le n{(\mathop {\max }\nolimits_{1 \le i \le n} |{a_i}|)^q} \le n(\sum\limits_{i = 1}^n {|{a_i}|{)^q}} \]
and
\[\left(\sum\limits_{i = 1}^n |{a_i}|\right)^q  \le {(n\mathop {\max }\nolimits_{1 \le i \le n} |{a_i}|)^q} \le {n^q}\sum\limits_{i = 1}^n {|{a_i}{|^q}}. \]
Thus, one has
\begin{eqnarray}\label{eqMMI_3}
E|{M_T}{|^q}& = &E{\left(\sqrt {\sum\limits_{i,j = 1}^p {|{M_{ij,T}}{|^2}}  } \right)^q} \nonumber\\
&\le& {C}(q)E{\left(\sum\limits_{i,j = 1}^p  {|{M_{ij,T}}|} \right)^q}\nonumber\\
 &\le& {C}(q)\sum\limits_{i,j = 1}^p  {E|{M_{ij,T}}{|^q} }.
\end{eqnarray}
Then it follows from (\ref{eqMMI_2}) and (\ref{eqMMI_3}) that (\ref{eqMMI_1}) holds. \hfill $\blacksquare$

\end{proof}

\begin{remark}\label{remark2.1}
Let $T>0$ be fixed, $B_t$ be a Brownian matrix (\cite{Bru}) in $\mathcal{M}_{p,p}( \mathbb{R})$ and $f_s(x)$ be a  random filed in $\mathcal{M}_{p,p}( \mathbb{R})$ such that
for any $t\in [0,T]$ and $x\in \Lambda \subseteq \mathbb{R}$,  $M_t=\int_0^t {{f_s}(x)d{B_s}}$ is a continuous local martingale, then we have
\begin{eqnarray*}
 \langle M{ \rangle _{ij,t}} &=& \sum\limits_{k = 1}^p  \langle  {M_{ik}},M{'_{kj}}{ \rangle _t} \\
 &=&\sum\limits_{k = 1}^p  \left\langle  \sum\limits_{u = 1}^p {\int_0^ \cdot  {{f_{iu,s}}(x)d{B_{uk,s}}} } ,\sum\limits_{v = 1}^p {\int_0^ \cdot  {{f_{jv,s}}(x)d{B_{vk,s}}} } \right\rangle _t\\
 &=& \sum\limits_{k = 1}^p {\sum\limits_{u = 1}^p {\int_0^t {{f_{iu,s}}(x){f_{ju,s}}(x)ds} } } \\
 &=& p{\left(\int_0^t {{f_s}(x){f_s}(x)} 'ds\right)_{ij}},
\end{eqnarray*}
which implies
\[ \langle M{ \rangle _t} = p\int_0^t {{f_s}(x){f_s}(x)} 'ds.\]
Moreover, for any $A, B\in \mathcal{M}_{p,p}( \mathbb{R})$, it is not hard to check that$|AB|\leq C|A||B|$.
For $q>2$, by above martingale moment inequality and H\"{o}lder's inequality, we have
\begin{eqnarray*}
E\left|\int_0^t {{f_s}(x)d{B_s}} \right|^q &\le& C(q)E\left|\int_0^t {{f_s}(x){f_s}(x)} 'ds \right|^{\frac{q}{2}}\\
 &\le& C(q)E{\left(\int_0^t {|{f_s}(x){f_s}(x)} '|ds\right)^{\frac{q}{2}}}\\
 &\le& C(q)E{\left(\int_0^t {|{f_s}(x)} {|^2}ds\right)^{\frac{q}{2}}} \\
 &\le& C(q)E\int_0^t {|{f_s}(x)} {|^q}ds
\end{eqnarray*}
and
\[E\left|\int_0^t {{f_s}(x)ds} \right|^q \le {C}(q)E\int_0^t {|{f_s}(x)} {|^q}ds,\]
which we shall use in Section \ref{section4}.
\end{remark}

\begin{lemma}\label{lemma2.2}
(Kolmogorov's Theorem) Let $f_s(x)=(f_{ij,s}(x))$ $(i,j=1,2,\cdots ,p; s\in [0,T];x\in \Lambda \subseteq \mathbb{R})$ be a measurable random (matrix) field in $\mathcal{M}_{p,p}( \mathbb{R})$ satisfying that $f_s(x)$ is $m$-times ($m\geq 1$) continuously differentiable in $x$ for all $s$ a.s. and that derivatives $\partial^kf_s(x)=(\partial^kf_{ij,s}(x))_{ij}$, $1\leq k\leq m$ fulfill the following conditions
\begin{eqnarray*}
&&\int_0^T {E[|{\partial^k}{f_s}(x){|^q}]} ds \le {C_1}(q),\\
&&\int_0^T {E[|{\partial^k}{f_s}(x) - {\partial^k}{f_s}(y){|^q}]} ds \le {C_2}(q)|x - y{|^{\gamma q}}
\end{eqnarray*}
for any $x,y\in\Lambda$ and $q>2$, where $0<\gamma\leq 1$ and ${C_1}(q)$, ${C_2}(q)$ are positive constants depending on $q$.
Then there are modifications of the integrals $\int_0^t {{f_s}(x)d{B_s}}$ and $\int_0^t {{f_s}(x)d{s}}$ which are continuous in $(t,x)$ and $m$-times continuously differentiable in $x$ for all $t$ a.s. Furthermore, it holds
\begin{equation*}
{\partial^k}\int_0^t {{f_s}(x)d{B_s}}  = \int_0^t {{\partial^k}{f_s}(x)d{B_s}},\quad {\partial^k}\int_0^t {{f_s}(x)d{s}}  = \int_0^t {{\partial^k}{f_s}(x)d{s}},
\end{equation*}
for $1\leq k\leq m$.
\end{lemma}
\begin{proof}
For any $i,j=1,2,\cdots,p$, $f_{ij,s}(x)$ is a measurable random field satisfying that $f_{ij,s}(x)$ is $m$-times continuously differentiable in $x$ for all $s$ a.s. and for any $x,y\in\Lambda$, $q>2$ and $0<\gamma\leq 1$, we have
$$
\int_0^T {E[|{\partial^k}{f_{ij,s}}(x){|^q}]} ds \le \int_0^T {E[|{\partial^k}{f_s}(x){|^q}]} ds \le {C_1}(q)
$$
and
$$
\int_0^T {E[|{\partial^k}{f_{ij,s}}(x) - {\partial^k}{f_{ij,s}}(y){|^q}]} ds \le \int_0^T {E[|{\partial^k}{f_s}(x) - {\partial^k}{f_s}(y){|^q}]} ds \le {C_2}(q)|x - y{|^{\gamma q}}.
$$
By Theorem 10.6 in \cite{Kunita82}, there are modifications of the integrals $\int_0^t {{f_{ij,s}}(x)d{B_{uv,s}}}$ and $\int_0^t {{f_{ij,s}}(x)ds}$  which are continuous at $(t,x)$ and $m$-times continuously differentiable in $x$ for all $t$ a.s.  Moreover, the modifications satisfy
\begin{equation*}
{\partial^k}\int_0^t {{f_{ij,s}}(x)d{B_{uv,s}}}  = \int_0^t {{\partial^k}{f_{ij,s}}(x)d{B_{uv,s}}},\quad {\partial^k}\int_0^t {{f_{ij,s}}(x)d{s}}  = \int_0^t {{\partial^k}{f_{ij,s}}(x)d{s}} \quad \forall k\le m.
\end{equation*}
As we have
\[\int_0^t {{f_s}(x)d{B_s}}  = {\left(\sum\limits_{u = 1}^p {\int_0^t {{f_{iu,s}}(x)d{B_{uj,s}}} } \right)_{1 \le i,j \le p}}\]
and
\[\int_0^t {{f_s}(x)d{s}}  = {\left({\int_0^t {{f_{ij,s}}(x)d{s}} } \right)_{1 \le i,j \le p}},\]
the lemma is easily proved. \hfill $\blacksquare$
\end{proof}

\begin{remark}\label{remark2.4}
 Let $D$ be the operator of Fr\'{e}chet derivative (\cite{Ciarlet}). It is obvious that $f_s:\mathbb{R}\rightarrow \mathbb{R}^{p\times p}$ is an abstract function and its derivative is a special case of the Fr\'{e}chet derivative. For any differentiable maps $g:\mathbb{R}^{p\times p}\rightarrow \mathbb{R}^{p\times p}$, by the chain rule of composite functions (Theorem 7.1-3 in \cite{Ciarlet}), we have
 \[\partial(g\circ f_s)(x)=Dg(f_s(x))(\partial f_s(x)),\]
 where $Dg(f_s(x))\in \mathcal{L}(\mathbb{R}^{p\times p}, \mathbb{R}^{p\times p})$ is the Fr\'{e}chet derivative of $g$ at $f_s(x)$. Similar representations can be used for higher derivatives. For simplicity, the norm of $Dg(f_s(x))$ is still denoted by $|Dg(f_s(x))|$.
\end{remark}

\begin{lemma}\label{ito}
(Generalized It\^{o} Formula) Let $F_t(x)$  ($x\in\mathbb{R}, t\in[0,T]$) be a random field in $\mathcal{M}_{p,p}( \mathbb{R})$ which is continuous in $(t,x)$ a.s. such that
\begin{itemize}
\item[$(a)$]$F_t(x)$ is twice continuously differentiable in $x$ a.s.;
\item[$(b)$]For every $x$, $F_t(x)$ is a following continuous semimartingale:
\[{F_t}(x) = {F_0}(x) + \int_0^t {{f_s}(x)} d{Y_s}, \; a.s. ,\]
where $Y_s$ is a continuous semimartingale in $\mathcal{M}_{p,p}( \mathbb{R})$ and ${f_s}(x)$ is a random field in $\mathcal{M}_{p,p}( \mathbb{R})$ which is  continuous in $(s,x)$ a.s. satisfying ${f_s}(x)$ is twice continuously differentiable adapted process for each $x$ a.s.
\end{itemize}
Let $X_t$ be a continuous semimartingale in $\mathbb{R}$. Then
\begin{eqnarray*}
{F_t}({X_t}) &=& {F_0}({X_0}) + \int_0^t {{f_s}({X_s})} d{Y_s} + \int_0^t {{\partial_x}{F_s}({X_s})} d{X_s}\\
 &&\mbox{}+ \int_0^t {{\partial_x}{f_s}({X_s})} d{\langle Y,X{I_p}\rangle _s} + \frac{1}{2}\int_0^t {\partial_x^2{F_s}({X_s})} d{\langle X\rangle _s}.
\end{eqnarray*}
\begin{proof}
It is obvious that, for $i,j=1,2,\cdots,p$,
\[{F_{ij,t}}(x) = {F_{ij,0}}(x) + \sum\limits_{k = 1}^p {\int_0^t {{f_{ik,s}}(x)} d{Y_{kj,s}}}. \]
By generalized It\^{o} formula (see Theorem 8.1 in \cite{Kunita82}), we have
\begin{eqnarray*}
{F_{ij,t}}({X_t}) &=& {F_{ij,0}}({X_0}) + \sum\limits_{k = 1}^p {\int_0^t {{f_{ik,s}}({X_s})} d{Y_{kj,s}}}  + \int_0^t {\partial_x}{F_{ij,s}}({X_s}) d{X_s}\\
 &&\mbox{} +\sum\limits_{k = 1}^p {\int_0^t \partial_x {f_{ik,s}}({X_s})} d\langle {Y_{kj}}, X{\rangle _s} + \frac{1}{2}\int_0^t \partial_x^2 {F_{ij,s}}({X_s}) d{\langle X\rangle _s},
\end{eqnarray*}
which is just what we need.\hfill $\blacksquare$
\end{proof}
\end{lemma}

\begin{definition}
(\cite{Mayerhofer}) Let $H$ be (a subset of) a normed space whose norm is denoted by $||\cdot||_H$ (if $H$ is $\mathcal{M}_{n,p}(\mathbb{R})$, then we still use $|\cdot|$).  A function $h:\mathbb{R}_+\times\mathcal{S}_{p}^+(\mathbb{R})\rightarrow H$ is said to be locally Lipschitz if
\[||h(t,X)-h(t,Y)||_H\leq C(U)|X-Y|\]
for any $t\in \mathbb{R}_+$, all compact sets $U\subseteq H$ and any $X,Y\in U$, where $C(U)$ is a constant depending on $U$. $h$ is said to be of linear growth if
\[||h(t,X)||_H^2\leq C(1+|X|^2)\]
for any $t\in \mathbb{R}_+$ and any $X\in \mathcal{S}_{p}^+(\mathbb{R})$.
\end{definition}

The following lemma is crucial in our paper.
\begin{lemma}\label{lem2.4}
(\cite{Mayerhofer}) Let $B_t$ be a Brownian matrix in $\mathcal{M}_{p,p}(\mathbb{R})$. Let $F,G: \mathbb{R}_+\times \mathcal{S}_{p}^+(\mathbb{R})\rightarrow \mathcal{M}_{p,p}(\mathbb{R})$ be two measurable functions such that $G'\otimes F(t,X)=(G(t,X))'\otimes F(t,X)$ is locally Lipschitz and of linear growth. Let $J: \mathbb{R}_+\times \mathcal{S}_{p}^+(\mathbb{R})\rightarrow \mathcal{S}_{p}(\mathbb{R})$ be locally Lipschitz and of linear growth. Assume that there exists a locally integrable function $c:\mathbb{R}_+\rightarrow \mathbb{R}$, i.e. $\int_0^a {|c(s)|ds}  <  + \infty $ for each $a\in \mathbb{R}_+$, such that
\begin{align}\label{lem2.4_eq1}
c(t)\leq &\mathrm{Tr}[J(t,X) X^{-1}]-\mathrm{Tr}[f(t,X) X^{-1}]\mathrm{Tr}[ g(t,X) X^{-1}]\nonumber\\
&\mbox{}-\mathrm{Tr}[f(t,X)X^{-1} g(t,X) X^{-1}]
\end{align}
for each $t\in \mathbb{R}_+$ and $X\in \mathcal{S}_{p}^+(\mathbb{R})$, where $f(t,X)=F(t,X) F(t,X)'$ and $g(t,X)=G(t,X)' G(t,X)$.

Then the stochastic differential equation (SDE)
\begin{equation}\label{lem2.4_eq2}
dX_t=F(t,X_t)dB_t G(t,X_t)+G(t,X_t)' dB_t' F(t,X_t)'+J(t,X_t)dt
\end{equation}
with $X_0=x\in \mathcal{S}_{p}^+(\mathbb{R}),a.s.$ has a unique adapted continuous strong solution $(X_t)_{t\in \mathbb{R}_+}$ on $\mathcal{S}_{p}^+(\mathbb{R})$.

In particular, the stopping time $T_x=\inf\{t\geq 0:X_t \notin \mathcal{S}_{p}^+(\mathbb{R})\}=+\infty,a.s.$.
\end{lemma}
\begin{proof}
The lemma easily follows from the Theorem 3.4 without jump diffusions in \cite{Mayerhofer}.\hfill $\blacksquare$
\end{proof}

\section{Fractional Wishart processes with integer indices}\label{section3}\noindent
\setcounter{equation}{0}
 Throughout, like the Brownian matrix, we shall define a fractional Brownian matrix $B_t^H$ with the Hurst index $H$ as a process taking its values in $\mathcal{M}_{n,p}( \mathbb{R})$ whose components are independent fractional Brownian motions, i.e.
\[B_t^H=(B_{ij}^H(t)),\quad B_0^H=(B_{ij}^H(0))=C, \; a.s.\]
where $C\in \mathcal{M}_{n,p}( \mathbb{R})$ is the initial state.

\begin{definition}\label{def3.1}
A fractional Wishart process, of Hurst index $H$, index $n$, dimension $p$ ($p\leq n$) and initial state $\Sigma_0$, written as $fWIS(H,n,p,\Sigma_0)$, is the matrix process
\begin{equation}\label{eq3.2}
{\Sigma _t} = ({\Sigma _{ij}}(t)) = {(B_t^H)'}B_t^H, \quad {\Sigma _0} = {C'}C>0, \; a.s.
\end{equation}
\end{definition}

\begin{remark}
Similar to the work in \cite{Bru}, we can investigate the process of eigenvalues of a fractional Wishart process with an integer index $n$ under some conditions. For more details, we refer to \cite{Pardo}.
\end{remark}

From Definition {\ref{def3.1}}, it is easy to see that a fractional Wishart process is a Wishart process in the sense of Bru (\cite{Bru}) when $H=\frac{1}{2}$ and we will see that a fractional Wishart process takes its values with probability one in $\mathcal{S}_{p}^+( \mathbb{R})$ and that for a fixed $t$, the random variable $\Sigma _t$ has a noncentral Wishart distribution.

\begin{lemma}\label{lemma3.1}
For a fixed $t>0$, \[\Sigma _t>0, \; a.s.\]
\[\  {\Sigma _t} \sim {W_p}(n,{t^{2H}}{I_p},{{t^{-2H}}\Sigma _0}).\]
Consequently, its p.d.f is given by
\[{\left\{ {2^{\frac{1}{2}np}}{\Gamma _p}\left(\frac{1}{2}n\right){t^{npH}}\right\} ^{ - 1}}\mathrm{etr}\left( - \frac{1}{2}{t^{ - 2H}}{\Sigma _0} - \frac{1}{2}{t^{ - 2H}}{\Sigma _t}\right)\det {({\Sigma _t})^{\frac{1}{2}(n - p - 1)}}{}_0{F_1}\left(\frac{1}{2}n;\frac{1}{4}{t^{ - 4H}}{\Sigma _0}{\Sigma _t}\right),\]
and its Laplace transform is given by
\begin{equation}\label{eq3.3}
E[\mathrm{etr}( - Z{\Sigma _t})] = \det {\left({I_p} + 2{t^{2H}}Z\right)^{ - \frac{n}{2}}}\mathrm{etr}\left( - Z{({I_p} + 2{t^{2H}}Z)^{ - 1}}{\Sigma _0}\right),
\end{equation}
for any $Z \in \mathcal{S}^+_{p}( \mathbb{R})$.
\end{lemma}

\begin{proof}
Firstly, for a fixed $t>0$, we give the distribution of the fractional Brownian matrix $B_t^H$.

As before, we denote $e_1=(1,0,\cdots,0)',\cdots,e_n=(0,\cdots,0,1)'$  in  $\mathcal{M}_{n,1}(\mathbb{R})$ and $f_1=(1,0,\cdots,0)'$, $\cdots$, $f_p=(0,\cdots,0,1)'$ in $\mathcal{M}_{p,1}(\mathbb{R})$. Then
\[\mathrm{vec}[(B_t^H)'] = \left( {\begin{array}{*{20}{c}}
{(B_t^H)'{e_1}}\\
 \vdots \\
{(B_t^H)'{e_n}}
\end{array}} \right)\]
and
\begin{eqnarray}\label{eq3.4}
&&\mathrm{cov} (\mathrm{vec}[(B_t^H)'])\nonumber\\
& =& E\{ \mathrm{vec}[(B_t^H)'] - E\mathrm{vec}[(B_t^H)']\} \{ \mathrm{vec}[(B_t^H)'] - E\mathrm{vec}[(B_t^H)']\} '\nonumber\\
 &=& E{((B_t^H-EB_t^H)'{e_i}{e_j}'(B_t^H-EB_t^H))_{1 \le i,j \le n}}
\end{eqnarray}
which is partitioned by $n^2$ matrices of $\mathcal{M}_{p,p}( \mathbb{R})$. For any $1 \le i,j \le n$ and $1 \le k,l \le p$, the $(k,l)^{th}$ component of the $(i,j)^{th}$ matrix of (\ref{eq3.4}) is calculated by the definition of fractional Brownian motion  as follows:
\begin{eqnarray*}
E({f_k}'(B_t^H - EB_t^H)'{e_i}{e_j}'(B_t^H - EB_t^H){f_l})= \left\{ {\begin{array}{ll}
0,& i \ne j\  or\  k \ne l;\\
{t^{2H}},& i = j,k = l.
\end{array}} \right.
\end{eqnarray*}
Thus, we have
\[ \mathrm{cov}(\mathrm{vec}[(B_t^H)']) = {I_n} \otimes ({t^{2H}}{I_p}).\]
Then it is not hard to check that
\[\mathrm{vec}[(B_t^H)'] \sim {N_{np}}(\mathrm{vec}(C'),{I_n} \otimes ({t^{2H}}{I_p})).\]
By the definition of the matrix variate normal distribution (see \cite{Gupta}), we obtain
\[B_t^H \sim {N_{n,p}}(C,{I_n} \otimes ({t^{2H}}{I_p})).\]

Finally, by Theorems 3.2.1 and 3.5.1 of \cite{Gupta}, it follows that
\[\Sigma _t>0,\quad {\Sigma _t} \sim {W_p}(n,{t^{2H}}{I_p},{{t^{-2H}}\Sigma _0}).\]
Then by Definition \ref{wd} and the proof of Theorem 3.5.1 in \cite{Gupta}, we can easily get its p.d.f and Laplace transform.  \hfill $\blacksquare$

\end{proof}

Similar to the Wishart processes, the fractional Wishart processes have the following additivity property.

\begin{theorem}
If $\Sigma_t$ and $S_t$ are two independent fractional Wishart processes $fWIS(H,n,p,\Sigma_0)$ and $fWIS(H,m,p,S_0)$, respectively, then $\Sigma_t+S_t$ is a fractional Wishart process
$$fWIS(H,n+m,p,\Sigma_0+S_0).$$
\end{theorem}

\begin{proof}
Assume that $\Sigma_t=(B_t^H)'B_t^H$ and $S_t=(W_t^H)'W_t^H$, where $B_t^H$ and $W_t^H$ are $n\times p$ and $m\times p$ independent fractional brownian matrices with initial state $B_0^H$ and $W_0^H$, respectively. If we set $E_t=(B_t',W_t')'$, then $E_t$ is an $(n+m)\times p$ matrix of independent fractional Brownian matrix with initial state $(B_0',W_0')'$ and we can check that
\[{\Sigma _t} + {S_t} = (B_t^H)'B_t^H + (W_t^H)'W_t^H = {E_t}'{E_t},\]
which just proves our statement. \hfill $\blacksquare$

\end{proof}

\begin{remark}\label{remark3.2}
By L\'{e}vy continuity Theorem (see, for example, Theorem 1.1.15 in \cite{Applebaum}), we know that (\ref{eq3.3}) is still a Laplace transform of a stochastic process whenever $n>0$ is no longer an integer. Nevertheless, we can not easily verify the properties of this process, for example, the property of positive definite. We shall still denote the stochastic process as $fWIS^*(H,v,p,\Sigma_0)$. When $n$ is a positive integer, it is easy to see that $fWIS^*(H,n,p,\Sigma_0)$ is an extension of   $fWIS(H,n,p,\Sigma_0)$ defined by (\ref{eq3.2}). Indeed, one has
$$
fWIS(H,n,p,\Sigma_0)\mathop  = \limits^d fWIS^*(H,n,p,\Sigma_0),
$$
where $\mathop  = \limits^d$  means that the two processes have same distribution for any $t>0$. However, we could not ensure
$$
fWIS(H,n,p,\Sigma_0)= fWIS^*(H,n,p,\Sigma_0),\quad a.s.
$$
\end{remark}

\begin{remark}
 Assume that the Hurst index $H\not=\frac{1}{2}$. As fractional Brownian motions exhibit the serial correlation, it is not hard to see that fractional Wishart processes also exhibit serial correlation, i.e. the increments of fractional Wishart processes depend on the information of the past. For example, taking $n=p=1$ in $fWIS(H,n,p,\Sigma_0)$,  we have
\[{\Sigma _t} - {\Sigma _s} = {(B_t^H)^2} - {(B_s^H)^2} = {(B_t^H - B_s^H)^2} + 2B_s^H(B_t^H - B_s^H)\]
for any $0<s<t$. As $B_t^H - B_s^H$ depends on the past $\mathcal{F}_s=\sigma\{B_u^H, 0\leq u\leq s\}$, we know that ${\Sigma _t} - {\Sigma _s}$ also depends on the past $\mathcal{F}_s$.
\end{remark}

\section{$\varepsilon$-Fractional Wishart processes}\label{section4}\noindent
\setcounter{equation}{0}

From \cite{Bru}, We know that the Wishart process is governed by a related SDE which can be used to describe some dynamic problems in applications. However, it seems to us that there is no way to develop SDE or SPDE to govern the fractional Wishart process directly.  Therefore, it is important to find some processes which not only maintain some properties of the fractional Wishart processes, but also can be governed by some related SDEs or SPDEs. To this end, in this section, we introduce the following $\varepsilon$-fractional Wishart processes which can be governed by some related SPDEs to approximate the fractional Wishart processes.

Define two processes in $\mathcal{M}_{n,p}( \mathbb{R})$ as follows:
\begin{equation}\label{eq4.1+}
B_t^{H,0} =\sqrt {2H}\int_0^t {{{(t - s)}^\alpha }} d{B_s}+C,
\end{equation}
\[B_t^{H,\varepsilon } = \sqrt {2H}\int_0^t {(t - s + \varepsilon )^\alpha} d{B_s} + C,\]
where $\alpha  = H - \frac{1}{2}$ and $B_t$ is a Brownian matrix in $\mathcal{M}_{n,p}( \mathbb{R})$ with initial state $C$ (i.e. $B_t=B_t^{\frac{1}{2}}$).

Next, we shall state two lemmas in the form of matrix from \cite{Thao} (for the case $\frac{1}{2}<H<1$) and \cite{Thao02} (for the case $0<H<\frac{1}{2}$). The proofs are trivial and so we omit them here.

\begin{lemma}\label{lemma4.1}
For any $\varepsilon>0$, $B_t^{H,\varepsilon }$ is a semimartingale satisfying
\begin{equation}\label{eq4.1}
dB_t^{H,\varepsilon } = \left(\sqrt {2H}\int_0^t {\alpha {{(t - s + \varepsilon )}^{\alpha  - 1}}} d{B_s}\right)dt + \sqrt {2H}{\varepsilon ^\alpha }d{B_t}.
\end{equation}
\end{lemma}

\begin{lemma}\label{lemma4.2}
$B_t^{H,\varepsilon }$ converges to $B_t^{H,0}$ in $L^2(\mathbf{\Omega})$ when $\varepsilon$ tends to $0$ and this convergence is uniform with respect to $t\in [0,T]$ for some $T\in \mathbb{R}$.
\end{lemma}

Note that (\ref{eq4.1+}) is defined in the sense of L\'{e}vy (\cite{Levy}), which is based on the Holmgren-Riemann-Liouville fractional integral while there is a definition in the sense of Mandelbrot-van Ness (\cite{Mandelbrot}) based on Weyl's integral for a two-side fractional Brownian motion:
\begin{equation}\label{eq4.1++}
B_t^H = C + {V_H}\left\{ {\int_{ - \infty }^0 {[{{(t - s)}^\alpha } - {{( - s)}^\alpha }]d{B_s}}  + \int_0^t {{{(t - s)}^\alpha }d{B_s}} } \right\}
\end{equation}
for $t>0$ (and similarly for $t<0$), where $V_H$ is a normalizing constant. Compared with the fractional Brownian motion $B_t^H$ defined by (\ref{eq4.1++}), the process  $B_t^{H,0}$ defined by (\ref{eq4.1+}) ``puts too great an importance on the origin for many applications" (\cite{Mandelbrot}). Furthermore, when $H\neq\frac{1}{2}$, it is easy to check that each component of $B_t^{H,0}$ does not satisfy the covariance condition in Definition \ref{d2.3} (or that the increment of each component of $B_t^{H,0}$ is not a stationary process which is contrary to the case of $B_t^H$, see \cite{Picard}) and so $B_t^{H,0}$ is not a fractional Brownian matrix.

 However, in the case $t>0$,  $\int_0^t {{{(t - s)}^\alpha }d{B_s}}$ (the main part of  $B_t^{H,0}$) plays an essential role in exhibiting the property of memory in $B_t^H$.   For any fixed $t>0$, we have $B_t^{H,0} \mathop  = \limits^d  B_t^H$, as we can compute that $B_t^{H,0} \sim {N_{n,p}}(C,{I_n} \otimes ({t^{2H}}{I_p}))$ by employing the same method used in Lemma \ref{lemma3.1}. Moreover, $B_t^{H,0}$ can be viewed as a good approximation of the fractional Brownian motion for large times (see Theorem 17 in \cite{Picard}). Hence, as a more simple process, there are large numbers of researchers adopt $B_t^{H,0}$ rather than the fractional Brownian motion defined by (\ref{eq4.1++}) to study SDEs related to the property of memory (see, for example, \cite{Alos,Jumarie,Thao02}). So we shall use $B_t^{H,\varepsilon}$ as the approximation of the fractional Brownian matrix $B_t^H$.

For any $\varepsilon>0$,  similar to Definition \ref{def3.1}, we can define the following  $\varepsilon$-fractional Wishart processes with integer indices.

\begin{definition}
For any $\varepsilon>0$, an $\varepsilon$-fractional Wishart process, of Hurst index $H$, index $n$, dimension $p(\leq n)$ and initial state $\Sigma_0^\varepsilon$, written as $\varepsilon$-$fWIS(H,n,p,\Sigma_0)$, is the matrix process
\begin{equation}\label{eq4.2}
{\Sigma _t^\varepsilon} = ({\Sigma _{ij}^\varepsilon}(t)) = {(B_t^{H,\varepsilon})'}B_t^{H,\varepsilon},\ \ {\Sigma _0^\varepsilon} = {C'}C >0, \; a.s.
\end{equation}
\end{definition}

Analogous to the fractional Wishart process, it is easy to get the following result.

\begin{theorem}
For any $\varepsilon>0$,
\[\Sigma _t^\varepsilon  > 0,\; a.s.\]
\[\Sigma _t^\varepsilon  \sim {W_p}\left(n,({(t + \varepsilon )^{2H}} - {\varepsilon ^{2H}}){I_p},{({(t + \varepsilon )^{2H}} - {\varepsilon ^{2H}})^{ - 1}}\Sigma _0^\varepsilon \right).\]
Consequently, its Laplace transform is given by
\begin{equation}\label{eq4.3}
E[\mathrm{etr}( - Z\Sigma _t^\varepsilon )] = \det {\left({I_p} + 2({(t + \varepsilon )^{2H}} - {\varepsilon ^{2H}})Z\right)^{ - \frac{n}{2}}}\mathrm{etr}\left( - Z{({I_p} + 2({(t + \varepsilon )^{2H}} - {\varepsilon ^{2H}})Z)^{ - 1}}\Sigma _0^\varepsilon \right)
\end{equation}
for any $Z \in \mathcal{S}^+_{p}( \mathbb{R})$.

\end{theorem}

\begin{proof}
Similar to the proofs of Lemma \ref{lemma3.1}, one has
\[B_t^{H,\varepsilon } \sim {N_{n,p}}(C,{I_n} \otimes (({(t + \varepsilon )^{2H}} - {\varepsilon ^{2H}}){I_p})).\]
Thus, we can easily get the results.  \hfill $\blacksquare$
\end{proof}

\begin{remark}\label{remark4.1}
From (\ref{eq3.3}) and (\ref{eq4.3}), one can get that $E[\mathrm{etr}( - Z\Sigma _t^\varepsilon )]$ converges to $E[\mathrm{etr}( - Z{\Sigma _t})]$ when $\varepsilon$ tends to $0$ and so $\Sigma _t^\varepsilon$ converges to $\Sigma _t$ in distribution when $\varepsilon$ tends to $0$. As a result, we can regard an $\varepsilon$-fractional Wishart process as an approximation of the fractional Wishart process in distribution.
\end{remark}

Similar to the case that the Wishart process is governed by a matrix SDE (\cite{Bru}), we shall see that an $\varepsilon$-fractional Wishart process can be governed by a matrix SPDE.

\begin{theorem}\label{th4.2}
Let $u_t(x)=\Sigma _t^x$ and ${\partial _x}{u_t}(x)=(\frac{\partial }{{\partial x}}{u_{ij,t}}(x))$.  Assume that $W_t$ is a matrix Brownian motion in $\mathcal{M}_{p,p}( \mathbb{R})$. Then, for any $0<\varepsilon<1$, $u_t(x)$ satisfies the following SPDE:
\begin{equation}\label{eq4.4}
d{u_t}(x) = \left({\partial _x}{u_t}(x) + 2nH{x^{2\alpha}}{I_p}\right)dt + \sqrt {2H}{x^\alpha}\left(\sqrt {{u_t}(x)} d{W_t} + d{W_t}'\sqrt {{u_t}(x)} \right)
\end{equation}
with the initial condition ${u_0}(x) = C'C>0$ a.s. for $C\in \mathcal{M}_{n,p}( \mathbb{R})$ and $x \in [\varepsilon,1]$ and with the boundary conditions ${u_t}(\varepsilon ) = \Sigma _t^\varepsilon$ and ${u_t}(1) = \Sigma _t^1$ a.s.
\end{theorem}

\begin{proof}
By matrix variate partial integration (see Lemma 5.11 in \cite{Barndorff-Nielsen}), the relation (\ref{eq4.2}) gives that, for any $x\in [\varepsilon,1]$,
\begin{equation}\label{eq4.5}
d\Sigma _t^x  = 2nH{x ^{2\alpha}}{I_p}dt + d(B_t^{H,x })'B_t^{H,x } + (B_t^{H,x })'dB_t^{H,x }.
\end{equation}
It follows from the definition of $B_t^{H,x }$ and Lemma \ref{lemma2.2} that, for any $x\in [\varepsilon,1]$,
\begin{eqnarray}\label{eq4.5+}
{\partial _x}B_t^{H,x} = \sqrt {2H}\int_0^t {{\partial _x}[{(t - s + x)}^\alpha ]} d{B_s} = \sqrt {2H}\int_0^t {\alpha {{(t - s + x )}^{\alpha  - 1}}} d{B_s}.
\end{eqnarray}
Now Lemma \ref{lemma4.1} shows that
\begin{equation}\label{eq4.6}
dB_t^{H,x} = {\partial _x}B_t^{H,x}dt + \sqrt {2H}{x^\alpha }d{B_t}.
\end{equation}
Taking (\ref{eq4.5}) and (\ref{eq4.6}) into account, we achieve
\begin{equation}\label{eq4.7}
d\Sigma _t^x  = (2nH{x ^{2\alpha}}{I_p} + {\partial _x}(B_t^{H,x})'B_t^{H,x } + (B_t^{H,x })'{\partial _x}B_t^{H,x})dt + \sqrt {2H}{x^\alpha }(d{B_t}'B_t^{H,x } + (B_t^{H,x })'d{B_t}).
\end{equation}
It is easy to check that
\begin{equation}\label{eq4.8}
{\partial _x}(B_t^{H,x})'B_t^{H,x} + (B_t^{H,x})'{\partial _x}B_t^{H,x} = {\partial _x}((B_t^{H,x})'B_t^{H,x}) = {\partial _x}\Sigma _t^x.
\end{equation}

Let $W_t$ be a matrix process satisfying
\begin{equation}\label{eq4.9}
d{W_t} = {(\Sigma _t^x)^{ - \frac{1}{2}}}(B_t^{H,x})'d{B_t}.
\end{equation}
By the L\'{e}vy martingale characterization of Brownian motion (\cite{Karatzas}), it is not hard to prove (one can refer to \cite{Pfaffel}) that $W_t$ is a matrix Brownian motion in $\mathcal{M}_{p,p}( \mathbb{R})$.

Combining (\ref{eq4.8})-(\ref{eq4.9}) with (\ref{eq4.7}), we know that (\ref{eq4.4}) is true. Obviously, the initial and boundary conditions are satisfied. This completes the proof. \hfill $\blacksquare$

\end{proof}

\begin{remark}\label{remark4+}
We would like to point out that Theorem \ref{th4.2} also holds for any $x\in [\varepsilon,a]$ with $\varepsilon< a <+\infty$.  In fact, for any $x\in [\varepsilon,a]$ with $\varepsilon< a <+\infty$, by Lemma \ref{lemma2.2}, we know that (\ref{eq4.5+}) is still true.
\end{remark}

In order to define the Wishart process with a non-integer index, Bru (\cite{Bru}) extended the SDE related to the Wishart process with an integer index to the case when the
index $n$ is non-integer. Following the idea, we propose to study SPDE (\ref{eq4.4}) to extend an $\varepsilon$-fractional Wishart process with an integer index $\varepsilon$-$fWIS(H,n,p,\Sigma_0)$ to the case when $n$ is not an integer.

Assume that $f(x)$ and $g(x)$ are two bounded functions in $C^{m,1}(\mathbb{R})$ with $m\geq 4$ (where $C^{m,\delta}(\mathbb{R})$ is the set of all $C^m$ functions whose partial derivatives $\partial ^{\alpha}$ with $\alpha\leq m$ are H\"{o}lder-continuous with exponent $\delta$).  We first extend (\ref{eq4.4}) to the following Cauchy problem:
\begin{equation}\label{eq4.10}
d{u_t}(x) = ({\partial _x}{u_t}(x) + vf(x){I_p})dt + g(x)\left(\sqrt {{u_t}(x)} d{W_t} + d{W_t}'\sqrt {{u_t}(x)} \right)
\end{equation}
with the initial condition $u_0(x)=\Sigma_0\in \mathcal{S}_{p}^+( \mathbb{R})$ a.s. such that (\ref{eq4.10}) coincides with (\ref{eq4.4}) when $v=n$ and $x\in [\varepsilon, 1]$.

For this aim, we can choose $f(x)$ and $g(x)$ as follows. Let $f(x)=g^2(x)$ and
\[g(x) = \left\{
\begin{array}{ll}
0, & x \le 0;\\
\sum\limits_{i = 1}^5 a_i x^{i + 4}, & 0 < x < \varepsilon; \\
\sqrt{2H}{x^\alpha}, & \varepsilon  \le x \le 1; \\
\sum\limits_{i = 1}^5 b_i(x - 2)^{i + 4}, & 1 < x < 2; \\
0, &2 \le x,
\end{array}
\right.\]
where $a_i$ and $b_i$ ($i=1,\cdots,5$) are determined by following equations
\[\left\{ {\begin{array}{l}
{\sum\limits_{i = 1}^5 {{a_i}{\varepsilon ^{i + 4}} = } \sqrt {2H}{\varepsilon ^\alpha }};\\
{\sum\limits_{i = 1}^5 {{a_i}(i + 4){\varepsilon ^{i + 3}} = } \alpha \sqrt {2H}{\varepsilon ^{\alpha  - 1}}};\\
{\sum\limits_{i = 1}^5 {{a_i}(i + 4)(i + 3){\varepsilon ^{i + 2}} = } \alpha (\alpha  - 1)\sqrt {2H}{\varepsilon ^{\alpha  - 2}}};\\
{\sum\limits_{i = 1}^5 {{a_i}(i + 4)(i + 3)(i + 2){\varepsilon ^{i+1}} = } \alpha (\alpha  - 1)(\alpha  - 2)\sqrt {2H}{\varepsilon ^{\alpha  - 3}}};\\
{\sum\limits_{i = 1}^5 {{a_i}(i + 4)(i + 3)(i + 2)(i + 1){\varepsilon ^{i}} = } \alpha (\alpha  - 1)(\alpha  - 2)(\alpha  - 3)\sqrt {2H}{\varepsilon ^{\alpha  - 4}}}
\end{array}} \right.\]
and
\[\left\{ {\begin{array}{*{20}{l}}
{\sum\limits_{i = 1}^5 {{{( - 1)}^{i + 4}}{b_i} = } \sqrt {2H}};\\
{\sum\limits_{i = 1}^5 {{{( - 1)}^{i + 3}}(i + 4){b_i} = } \alpha \sqrt {2H}};\\
{\sum\limits_{i = 1}^5 {{{( - 1)}^{i + 2}}(i + 4)(i + 3){b_i} = } \alpha (\alpha  - 1)\sqrt {2H}};\\
{\sum\limits_{i = 1}^5 {{{( - 1)}^{i + 1}}(i + 4)(i + 3)(i + 2){b_i} = } \alpha (\alpha  - 1)(\alpha  - 2)\sqrt {2H}};\\
{\sum\limits_{i = 1}^5 {{{( - 1)}^{i    }}(i + 4)(i + 3)(i + 2)(i + 1){b_i} = } \alpha (\alpha  - 1)(\alpha  - 2)(\alpha  - 3)\sqrt {2H}}.
\end{array}} \right.\]
By the knowledge of the system of linear equations (see, for example, \cite{Singh}), we know that two systems of linear equations mentioned above have solutions.  Thus, we can find $f(x)$ and $g(x)$ such that (\ref{eq4.10}) coincides with (\ref{eq4.4}) when $v=n$ and $x\in [\varepsilon, 1]$. It is seen that $\partial_x^kf(x)$ and $\partial_x^kg(x)$ ($0\leq k\leq 4$) are continuous and bounded on $\mathbb{R}$.

The stochastic characteristic equations associated with (\ref{eq4.10}) are defined by
\begin{eqnarray}\label{eq4.12}
\left\{\begin{array}{l}
d{\xi _t} =  - dt,\quad {\xi _0} = x \in \mathbb{R}, \; a.s.\\
d{\eta _t} = vf({\xi _t})I_pdt + g({\xi _t})(\sqrt {{\eta _t}} d{W_t} + d{W_t}'\sqrt {{\eta _t}} ),\quad {\eta _0} = {\Sigma _0} \in \mathcal{S}_{p}^+( \mathbb{R}),\; a.s.
\end{array}\right.
\end{eqnarray}
Noticing that both $\xi _t$ and $\eta _t$ depend on $x$, we use $\xi _t(x)$ and $\eta _t(x)$ to denote the relationships if it is necessary.

For the solution of (\ref{eq4.12}), we have the following result.

\begin{lemma}\label{lemma4.3}
If $v\geq n+1$, then (\ref{eq4.12}) has a unique adapted continuous strong solution $(\xi _t,\eta _t)_{t\in \mathbb{R}_+}$ on $\mathbb{R}\times\mathcal{S}_{p}^+(\mathbb{R})$ for each initial condition $(x,\Sigma _0)\in \mathbb{R}\times\mathcal{S}_{p}^+(\mathbb{R})$. In particular, the stopping time $T(x,\Sigma _0)=\inf\{t\geq 0:\eta _t \notin \mathcal{S}_{p}^+(\mathbb{R})\}=+\infty,a.s.$.
\end{lemma}
\begin{proof}
It is telling that the first equation of (\ref{eq4.12}) has a unique adapted continuous strong solution:
$$
\xi_t(x)=x-t.
$$
So we focus on the second equation of (\ref{eq4.12}).

From the definitions of $f(x)$ and $g(x)$, there exists a real number $M>0$ such that
\begin{equation}\label{eq4.15}
\mathop {\sup }\limits_{x \in \mathbb{R}} \{ |f(x)|,|g(x)|\}  \le M
\end{equation}

Let $J(t,X)=vf(x-t)I_p$, $G(t,X)=g(x-t)I_p$, $F(t,X)=\sqrt{X}$, $ff(t,X)=F(t,X)F(t,X)'$ and $gg(t,X)=G(t,X)' G(t,X)$  for any $X\in\mathcal{S}_{p}^+(\mathbb{R})$. Then we have $ff(t,X)=X$ and $gg(t,X)=f(x-t)I_p$ for any $X\in\mathcal{S}_{p}^+(\mathbb{R})$.

For any $A\in \mathcal{M}_{m,n}(\mathbb{R})$ and $B\in \mathcal{M}_{p,q}(\mathbb{R})$,
\[|A\otimes B|^2 = \sum\limits_{i = 1}^m {\sum\limits_{j = 1}^n {\sum\limits_{u = 1}^p {\sum\limits_{v = 1}^q {|{A_{ij}}{B_{uv}}|^2} } } }  = \sum\limits_{i = 1}^m {\sum\limits_{j = 1}^n {|{A_{ij}}|^2} } \sum\limits_{u = 1}^m {\sum\limits_{v = 1}^n {|{B_{uv}}|^2} }  = |A|^2|B|^2\]
and it follows that, for any compact set $U$ in $\mathcal{S}_{p}^+(\mathbb{R})$ and $X,Y\in U$,
\[|G'\otimes F(t,X)-G'\otimes F(t,Y)|=|g(x-t)I_p\otimes (\sqrt{X}-\sqrt{Y})|\leq |g(x-t)||I_p||\sqrt{X}-\sqrt{Y}|.\]
Since $\sqrt{X}$ is locally Lipschitz in $\mathcal{S}_{n}^+(\mathbb{R})$ (see, for example, \cite{Barndorff-Nielsen}), there exists a constant $C(U)$ such that
\[|\sqrt{X}-\sqrt{Y}|\leq C(U)|X-Y|.\]
Then by (\ref{eq4.15}), we obtain
\begin{equation}\label{th1_eq1}
|G'\otimes F(t,X)-G'\otimes F(t,Y)|\leq M|I_p|C(U)|X-Y|
\end{equation}
which implies $G'\otimes F(t,X)$ is locally Lipschitz.

Now consider the norm $||\cdot||_2$ for $A\in\mathcal{M}_{m,n}(\mathbb{R})$ (see P44 in \cite{Dennis}):
\[||A||_2=(\mbox{maximum eigenvalue of } A'A)^{\frac{1}{2}}.\]
Then for any $Y\in \mathcal{S}_{p}^+(\mathbb{R})\subseteq \mathcal{M}_{p,p}(\mathbb{R})$, we have
\[||\sqrt{Y}||_2^2=||Y||_2.\]
Since $\mathcal{M}_{p,p}(\mathbb{R})$ is a $p^2$-dimensional space, by the equivalence of $|\cdot|$ and $||\cdot||_2$, there exists a constant $C$ such that
\[|\sqrt{Y}|^2\leq C|Y|.\]
Then by (\ref{eq4.15}) again, we get
\[|G'\otimes F(t,Y)|^2\leq |g(x-t)I_p|^2|\sqrt{Y}|^2\leq M^2pC|Y|\leq M^2pC(1+|Y|^2)\]
which implies $G'\otimes F(t,Y)$ is of linear growth.

It is obvious that $J(t,X)=vf(x-t)I_p$ is locally Lipschitz and of linear growth. Then by Lemma \ref{lem2.4} it suffices to check condition (\ref{lem2.4_eq1}) to complete the proof.

Indeed, we have
\begin{align}
&\mathrm{Tr}[J(t,X) X^{-1}]-\mathrm{Tr}[ff(t,X) X^{-1}]\mathrm{Tr}[ gg(t,X) X^{-1}]\nonumber\\
&\mbox{}-\mathrm{Tr}[ff(t,X) X^{-1}gg(t,X) X^{-1}]\nonumber\\
=&\mathrm{Tr}[(v-p-1)f(t-x)X^{-1}]\nonumber\\
\geq &0
\end{align}
and choose $c(t)=0$. Then the condition (\ref{lem2.4_eq1}) is satisfied and so the lemma holds.\hfill $\blacksquare$
\end{proof}

\begin{lemma}\label{lemma4.5}
Let $\eta_t$ satisfy the second equation of (\ref{eq4.12}). For each $t\in [0,T]$ , $\eta_t(x)$ is a $C^3$-map from $\mathbb{R}$ into $\mathcal{M}_{p,p}(\mathbb{R})$ a.s., i.e. $\eta_{ij,t}(x)$ $(i,j=1,2,\cdots,p)$ is a $C^3$-map from $\mathbb{R}$ into $\mathbb{R}$ a.s.
\end{lemma}
\begin{proof}
Firstly, from the proof of Theorem  4.11  of \cite{Pfaffel}, we know that $\mathcal{S}_{p}^+( \mathbb{R})$ is open in $\mathcal{S}_{p}( \mathbb{R})$ and there exist closed convex sets $U_n\in\mathcal{S}_{p}^+( \mathbb{R})$ such that $U_n\subseteq U_{n+1}$ and $\bigcup\nolimits_{n = 1}^{ + \infty } {{U_n}}  =\mathcal{S}_{p}^+( \mathbb{R}) $.
Let
$${V_n} = {U_n} \cap \{ X \in \mathcal{S}_{p}^+( \mathbb{R}): |X| \le n\}, \quad n=1,2,\cdots.$$
Then it is easy to see that  $V_n$ is compact and convex such that
$$
V_n\subseteq V_{n+1}, \quad \bigcup\limits_{n = 1}^{ + \infty } {{V_n}}  =\mathcal{S}_{p}^+( \mathbb{R}).
$$
Define $I_n=\{x\in \mathbb{R}: |x|<n\}$ and let $I$ be a finite open interval. For all fixed $t\in [0,T]$ and $x\in I \subseteq\mathbb{R}$, we can find a positive integer $N$ such that $x\in I_N$ and $\xi_t(x)=x-t\in I_N$. As the map $h:\mathcal{S}_{p}^+( \mathbb{R})\to \mathcal{S}_{p}^+( \mathbb{R})$ defined by  $A\mapsto \sqrt A$ is analytic (see P134 in \cite{Rogers1987} or Section 7.14 in \cite{Ciarlet}), we can choose $h^N(\xi,A):\mathbb{R}\times \mathcal{S}_{p}^+( \mathbb{R})\rightarrow \mathcal{S}_{p}( \mathbb{R})$ such that $h^N(\xi,A)$ is globally  Lipschitz in $\mathbb{R}\times \mathcal{S}_{p}^+( \mathbb{R})$ such that,  if $\xi\in I_N$ and $A\in V_N$,
then
\begin{equation}\label{eq4.19+}
h^N(\xi,A)=g(\xi)h(A)=g(\xi)\sqrt A
\end{equation}
(see Remark \ref{extension}). Assume that the Lipschitz constant of $h^N$ and $f$ is $K$, i.e. for any $(\xi_1,A_1), (\xi_2,A_2) \in \mathbb{R}\times \mathcal{S}_{p}^+( \mathbb{R})$,
\[
\left|{h^N}({\xi _1},{A_1}) - {h^N}({\xi _2},{A_2})\right| \le K(|{\xi _1} - {\xi _2}| + |{A_1} - {A_2}|)
\]
and
\[|f(\xi _1)-f(\xi _2)|\leq K |\xi _1-\xi _2|.\]

Let $X^N_s=({\xi _s^N},{\eta ^N_s})\in \mathbb{R} \times \mathcal{S}_{p}( \mathbb{R})$ satisfy the following SDEs:
\begin{eqnarray}\label{eq4.19}
\left\{
\begin{array}{l}
d{\xi _s^N} =  - ds,\quad {\xi^N _0} = x \in \mathbb{R},\; a.s.\\
d\eta _s^N = vf({\xi _s^N})I_pds + h^N(X^N_s)d{W_s} + d{W_s}'h^N({X^N_s}),\quad {\eta _0} = {\Sigma _0}\in \mathcal{S}_{p}^+( \mathbb{R}),\; a.s.
\end{array}
\right.
\end{eqnarray}
It follows from Theorem 1.2 in Chapter II in \cite{Kunita82} that there exists a unique strong solution $X^N_s$ on $s\in[0,+\infty)$. Further, $\eta _s^N(x)$ is in $L^q$ for any $q\geq 1$. The solution $X^N_s$ must be equal to $X_s=({\xi _s},{\eta_s })$ as long as it stays in $I_N\times V_N$, and vice versa, as $h^N(\xi,A)=g(\xi)\sqrt A$ on $I_N\times V_N$. By the arbitrariness of $x$, it suffices to check that $\eta_t^N(\cdot)$ is a $C^3$-map a.s.

Let $X_0=(\xi_0,A_0)$ be fixed in $I_N\times V_N$. Since $\eta _s^N(x)$ is in $L^q$ for any $q\geq 1$, it is not hard to prove that for any $y\; ,z\in I$ and $u\in [0,T]$, the integrals
$$
\int_0^u {[h^N(X_s^N(z))} - h^N(X_s^N(y))]d{W_s},\quad  \int_0^u {{h^N}(X_s^N(y))d{W_s}}
$$
are continuous martingales.

By Lemma \ref{MMI} and Remark \ref{remark2.1}, for any $q>2$, $y,\;z\in I$ and $u\in [0,T]$, we have
\begin{eqnarray}\label{eq4.21}
&&E|\eta_u^N(z) - \eta_u^N(y){|^q}\nonumber\\
 &=& E\left(\left|\int_0^u {v[f(\xi _s^N(z)) - f(\xi _s^N(y))} ]I_pds + \int_0^u {[h^N(X_s^N(z))}  - h^N(X_s^N(y))]d{W_s}\right.\right.\nonumber\\
 &&\left.\left.\mbox{}+ \int_0^u {d{W_s}'[h^N(X_s^N(z))}  - h^N(X_s^N(y))]\right|\right)^q\nonumber\\
 &\le& {C}(q)E\left( \left|\int_0^u v[f(\xi _s^N(z)) - f(\xi _s^N(y)) ]ds\right|^q + 2\left|\int_0^u {[h^N(X_s^N(z))}  - h^N(X_s^N(y))]d{W_s}\right|^q\right)\nonumber\\
 &\le& {C}(q)E\left(\int_0^u {|f(\xi _s^N(z)) - f(\xi _s^N(y))} {|^q}ds + \int_0^u {|h^N(X_s^N(z))}  - h^N(X_s^N(y)){|^q}ds\right)\nonumber\\
 &\le& {C}(q)E\left(K^q\int_0^u {|\xi _s^N(z) - \xi _s^N(y)} {|^q}ds + {K^q}\int_0^u {|X_s^N(z))}  - X_s^N(y){|^q}ds\right)\nonumber\\
 &\le& C(q)\left(|z - y{|^q} + \int_0^u {E|\eta_s^N(z))}  - \eta_s^N(y){|^q}ds\right).
\end{eqnarray}
By Gronwall's inequality (see, for example, Problem 5.2.7 in \cite{Revuz}), it follows from (\ref{eq4.21}) that
\begin{equation}\label{eq4.22}
E|\eta_u^N(z) - \eta_u^N(y){|^q} \le {C}(q)|z - y{|^q}.
\end{equation}
For any $u\in [0,T]$ and $x\in I$, by Lemma \ref{MMI} and Remark \ref{remark2.1}, we have
\begin{eqnarray*}
&&E|\eta_u^N(x){|^q}\\
&\le& C(q)E\left(1+\left|\int_0^u {vf(\xi _s^N(x)){I_p}ds} \right|^q + \left|\int_0^u {{h^N}(X_s^N(x))d{W_s}} \right|^q\right)\\
&\le& {C}(q)\left(1 + E\int_0^u {|{h^N}(X_s^N(x)){|^q}ds} \right)\\
&\le& {C}(q)\left(1 + E\int_0^u {{{\left(|{h^N}(X_0)| + K(|{\xi _0} - \xi _s^N(x)| + |{A_0} - \eta _s^N(x)|)\right)}^q}ds} \right)\\
&\le& {C}(q)\left(1 + \int_0^u {E|\eta_s^N(x){|^q}ds} \right)
\end{eqnarray*}
which implies, by Gronwall's inequality, that
\begin{equation}\label{eq4.24}
E|\eta_u^N(x)|^q\leq C(q).
\end{equation}
Now by (\ref{eq4.22}) and (\ref{eq4.24}), one can easily check that all the conditions of Lemma \ref{lemma2.2} are satisfied. Thus, by Lemma \ref{lemma2.2}, we know that  $\eta_u^N(z)$ is continuous at $(t,x)\in [0,T]\times I$ a.s. By the arbitrariness of $t$ and $x$, $\eta_u(z)$ is continuous at $(t,x)\in [0,T]\times \mathbb{R}$ a.s.

It is obvious that $f$, $g$ on $I_{N}$ and $\sqrt A$ on $V_{N}$ are $4$-times continuously differentiable in each variable and all their $k$-th  ($k\leq 4$) derivatives are bounded. In order to prove $\eta_t^N(x)$ is continuously differentiable at $x$ for $t\in[0,T]$, a.s., we need more conditions on $h^N(\cdot)$ defined by (\ref{eq4.19+}). Here we extend  $h^N(\xi,\eta)=g(\xi)h^*(\eta)$ (see Remark \ref{extension}) satisfying that $h^*:\mathcal{S}_p(\mathbb{R})\rightarrow\mathcal{S}_p(\mathbb{R})$ coincides with $h(\cdot)=\sqrt{\cdot}$ on $V_{N}$ and that there exists a constant $Q>0$ such that
\begin{equation}\label{eq4+3}
\sup \{|\partial_\xi^mg(\xi)|,|D_\eta^lh^*(\eta)|: \; m,l=0,\cdots,4; \xi\in \mathbb{R}, \eta\in \mathcal{S}_p(\mathbb{R})\}  \le Q.
\end{equation}
Denote the derivatives of $h^N(\cdot)$ by
\[\partial_\xi^mD_\eta^lh^N(\xi,\eta)=\partial_\xi^mg(\xi)D_\eta^lh^*(\eta)\]
for any $\xi\in\mathbb{R}$ and $\eta\in\mathcal{S}_p(\mathbb{R})$. Then it is seen that $\partial_\xi^mD_\eta^lh^N(\xi,\eta)$ are globally  Lipschitz and bounded in $\mathbb{R}\times \mathcal{S}_{p}^+( \mathbb{R})$ for $m,l\leq 3$.

Let
\[ R_u^\eta (z,y) = \frac{1}{y}(\eta _u^N(z + y) - \eta _u^N(y))\]
for any $z,z+y\in I$.
In order to prove that $ \eta_u^N(z)$ is differentiable at $z=x\in I$ a.s., it suffices to check that ${R_u^\eta}(x,y)$ is continuous at $y=0$ a.s.

For any $q>2$, by Lemma \ref{MMI} and Remark \ref{remark2.1}, we have
\begin{eqnarray*}
&&E|{R_u^\eta}(z,y){|^q}\\
 &\le& {C}(q)E\left(\left|\int_0^u v\frac{1}{y}[f(\xi _s^N(z + y)) - f(\xi _s^N(y)) ]ds\right|^q \right.\\
 &&\left.\mbox{}+ 2\left|\int_0^u \frac{1}{y}[h^N(X_s^N(z + y)) - h^N(X_s^N(y))]d{W_s}\right|^q\right)\\
 &\le& {C}(q)E\left(\int_0^u {{(vK)}^q}\left|\frac{1}{y}(\xi _s^N(z + y) - \xi _s^N(y) )\right|^qds + 2\int_0^u {K^q}\left|\frac{1}{y}(X_s^N(z + y)  - X_s^N(y))\right|^qds\right)\\
 &\le& {C}(q)\left(1 + \int_0^u {E|{R_s^\eta}(z,y){|^q}} ds\right),
\end{eqnarray*}
which implies, by Gronwall's inequality, that
\begin{equation}\label{eq4.25}
E|{R_u^\eta}(z,y){|^q} \le {C}(q).
\end{equation}
By (\ref{eq4.22}), (\ref{eq4+3}), (\ref{eq4.25}), Lemma \ref{MMI} and Remark \ref{remark2.1}, for any $q>2$, we get
\begin{eqnarray*}
&&E\bigg|\int_0^u {\bigg\{ \frac{1}{y}[{h^N}(\xi _s^N(x + y),\eta _s^N(x + y)) - {h^N}(\xi _s^N(x),\eta _s^N(x))]} \\
&& \mbox{}- \frac{1}{z}\left[{h^N}(\xi _s^N(x + z),\eta _s^N(x + z)) - {h^N}(\xi _s^N(x),\eta _s^N(x))\right]\bigg\} d{W_s}{\bigg|^q}\\
 &=& E\bigg|\int_0^u {\bigg\{ \frac{1}{y}\left[{h^N}(\xi _s^N(x + y),\eta _s^N(x + y)) - {h^N}(\xi _s^N(x + y),\eta _s^N(x))\right]} \\
&& \mbox{}+ \frac{1}{y}\left[{h^N}(\xi _s^N(x + y),\eta _s^N(x)) - {h^N}(\xi _s^N(x),\eta _s^N(x))\right]\\
&& \mbox{}- \frac{1}{z}\left[{h^N}(\xi _s^N(x + z),\eta _s^N(x + z)) - {h^N}(\xi _s^N(x + z),\eta _s^N(x))\right]\\
&& \mbox{}- \frac{1}{z}\left[{h^N}(\xi _s^N(x + z),\eta _s^N(x)) - {h^N}(\xi _s^N(x),\eta _s^N(x))\right]\bigg\} d{W_s}{\bigg|^q}\\
 &=& E\bigg|\int_0^u \bigg\{  \int_0^1 {{D_\eta }h^N(\xi _s^N(x + y),\eta _s^N(x) + w(\eta _s^N(x + y) - \eta _s^N(x)))}(R_s^\eta (x,y)) dw\\
&& \mbox{}+ \int_0^1 {{\partial _\xi }h^N(\xi _s^N(x) + w(\xi _s^N(x + y) - \xi _s^N(x)),\eta _s^N(x))} dw\\
&& \mbox{}-   \int_0^1 {{D _\eta }h^N(\xi _s^N(x + z),\eta _s^N(x) + w(\eta _s^N(x + z) - \eta _s^N(x)))}({ R_s^\eta (x,z)}) dw\\
&& \mbox{}- \int_0^1 {{\partial _\xi }h^N(\xi _s^N(x) + w(\xi _s^N(x + z) - \xi _s^N(x)),\eta _s^N(x))} dw\bigg\} d{W_s}{\bigg|^q}\\
 &\le& C(q)E\bigg\{ \int_0^u \bigg|
\int_0^1 {D_\eta }h^N(\xi _s^N(x + y),\eta _s^N(x) + w(\eta _s^N(x + y) - \eta _s^N(x)))\\
  &&\mbox{}\circ(R_s^\eta (x,y) - R_s^\eta (x,z)) dw{\bigg|^q}ds \\
&& \mbox{}+ \int_0^u \bigg|\int_0^1 \left[{D_\eta }h^N(\xi _s^N(x + y),\eta _s^N(x) + w(\eta _s^N(x + y) - \eta _s^N(x)))\right.  \\
&& \mbox{}- \left.{D_\eta }h(\xi _s^N(x + z),\eta _s^N(x) + w(\eta _s^N(x + z) - \eta _s^N(x)))\right](R_s^\eta (x,z))dw{\bigg|^q}ds\\
&& \mbox{}+ \int_0^u \bigg|\int_0^1 \left[{\partial _\xi }h^N(\xi _s^N(x) + w(\xi _s^N(x + y) - \xi _s^N(x)),\eta _s^N(x))\right. \\
&& \mbox{}- \left.{\partial _\xi }h^N(\xi _s^N(x) + w(\xi _s^N(x + z) - \xi _s^N(x)),\eta _s^N(x))\right]dw {\bigg|^q}ds\bigg\} \\
 &\le& C(q)\left[E\int_0^u {|R_s^\eta (x,y) - R_s^\eta (x,z){|^q}ds}  + |y - z{|^q}\right],
\end{eqnarray*}
where we used the chain rule in Remark \ref{remark2.4} and Taylor formula with integral remainder (see Theorem 7.9-1 in \cite{Ciarlet}). Similarly, we have
\begin{eqnarray*}
&&E\left|\int_0^u \frac{1}{y}[f(\xi _s^N(x + y)) - f(\xi _s^N(x))]ds - \int_0^u \frac{1}{z}[f(\xi _s^N(x + z)) -  f(\xi _s^N(x))]\right|^q\\
 &\le& {C}(q)\left(\int_0^u {E|{R_s^\eta}(x,y) - {R_s^\eta}(x,z)} {|^q}ds + |y - z{|^q}\right).
\end{eqnarray*}
Consequently, it follows that
\begin{equation*}
E|R_u^\eta (x,y) - R_u^\eta (x,z)|^q
\leq{C}(q)\left(\int_0^u {E|{R_s^\eta}(x,y) - {R_s^\eta}(x,z)} {|^q}ds + |y - z{|^q}\right),
\end{equation*}
which implies, by Gronwall¡¯s inequality, that
\begin{equation}\label{eq4.26}
E|{R_u}(x,y) - {R_u}(x,z){|^q} \le {C}(q)|y - z{|^q}.
\end{equation}

Note that in order to prove (\ref{eq4.25}) and (\ref{eq4.26}), we use the continuous martingale property of the integrals with respect to $W_t$ in above inequalities which is easy to prove as $\eta _s^N(x)$ is in $L^q$ for any $q\geq 1$.

Now by (\ref{eq4.25}) and (\ref{eq4.26}), it is easy to see that all the conditions of Lemma \ref{lemma2.2} are satisfied. Thus, by Lemma \ref{lemma2.2}, we conclude that ${R_u^\eta}(x,y)$ is continuous at $y=0$ a.s. and so $\eta_u^N(z)$ is continuously differentiable at $x\in I$ for any $u\in [0,T]$ a.s.

We can prove the claim that $\eta_u^N(z)$ is $k$-times ($k=2,3$) continuously differentiable at $x$ for any $u\in [0,T]$ a.s. in a similar way by considering the related SDEs and (\ref{eq4+3}). For example, when $k=2$, the related SDEs are given as follows:
\begin{eqnarray}\label{eq4.28}
\left\{\begin{array}{l}
d{\partial_y}{\xi _u^N}(y) = 0,\\
d{\partial_y}\eta _u^N(y)  =  v{\partial_\xi }f({\xi _u^N(y)}){\partial_y}{\xi _u^N}(y)I_pdu \\
 \qquad\qquad\quad \mbox{}+ {\partial_\xi }g({\xi _u^N(y)}){\partial_y}{\xi _u^N}(y)\left[h^*(\eta _u^N(y))d{W_u} + d{W_u}'h^*(\eta _u^N(y))\right]\\
 \qquad\qquad\quad \mbox{}+ g({\xi _u^N(y)})\left[{D_\eta }h^*(\eta _u^N(y))({\partial_y}\eta _u^N(y))d{W_u} + d{W_u}'{D_\eta }h^*(\eta _u^N(y))({\partial_y}\eta _u^N(y))\right],
\end{array}\right.
\end{eqnarray}
with the initial conditions  ${\partial_y}{\xi _0^N}(y)=1$ and ${\partial_y}\eta _0^N(y)=0$ a.s.

Therefore, by the discussions mentioned above and the arbitrariness of $x$ and $t$, we obtain that $\eta_t(x)$ is a $C^3$-map a.s. for each $t\in [0,T]$. \hfill $\blacksquare$

\end{proof}

\begin{remark}\label{extension}
Here we explain the feasibility of the extension in (\ref{eq4+3}). Owing to the definition of $g(x)$, it suffices to extend $h(A)=\sqrt{A}$ ($A\in V_N$) to $h^*(A)$ ($A\in\mathcal{S}_p(\mathbb{R})$) such that $h^*(A)=h(A)$ for $A\in V_N$ and that $D^kh(A)$ is continuous and bounded for $k\leq 4$ and $A\in\mathcal{S}_p(\mathbb{R})$. Without loss of generality, we just prove the case of $k\leq 1$. By extension of Tietze's theorem in \cite{Dugundji}, $D^1h(A)$ ($A\in V_N$) can be extended to a continuous map $h_1$ on $\mathcal{S}_p(\mathbb{R})$. Making use of Urysohn's lemma (see Lemma 15.6 in \cite{Willard}), there exists a continuous function $h_2$ defined on $\mathcal{S}_p(\mathbb{R})$ such that
\begin{equation*}
h_2(A)=\left\{\begin{array}{ll}
1,&A\in V_N;\\
0,&A\in \mathcal{S}_p(\mathbb{R})  \backslash C,
\end{array}\right.
\end{equation*}
where $C$ is an open convex set in $\mathcal{S}_p(\mathbb{R})$ satisfying that $A_N\subseteq C$ and $0\in \bar{C}\backslash C $ ($\bar{C}$ is the closure of $C$). Define $h_0(A)=h_1(A)h_2(A)$ for $A\in\mathcal{S}_p(\mathbb{R})$ and it is easy to see that $h_0$ is continuous and bounded on $\mathcal{S}_p(\mathbb{R})$. Now define $h^*$ as follows
\begin{equation*}
h^*(A)=\int_0^1 h_0(sA)(A)ds, \quad A\in \mathcal{S}_p(\mathbb{R}).
\end{equation*}
Let $y(t)=h^*(tA)=\int_0^t h_0(sA)(A)ds$ be an abstract function on $[0,1]$. Then using the continuity of $h_0$, it is easy to see
\[\frac{dy(t)}{dt}=h_0(tA)(A),\]
which implies $D^1h^*$ exists. Since the chain rule gives
\[\frac{dy(t)}{dt}=D^1h^*(tA)(A),\]
we obtain $D^1h^*=h_0$. Consequently, we obtain the feasibility of the extension in (\ref{eq4+3})
\end{remark}

\begin{remark}\label{remark4.4+}
By the arbitrariness of $I$ and $I_N\times V_N$, since we have that $\eta_t(x)$ coincides with $\eta^N_t(x)$ if both of the solutions stay in $V_N$, $\eta_t(x)$ should also satisfies the second equation of (\ref{eq4.28}).
\end{remark}

Let ${\bar \eta _t}(x) = {\eta _t}(x,{\Sigma _0})$. By (\ref{eq4.12}), it is easy to see that ${\xi _t}(x) = x - t$ and so $\xi _t^{ - 1}(x) = x + t$.
Let
$$u_t(x)={{\bar \eta }_t}(\xi _t^{ - 1}(x))={{\bar \eta }_t}(x + t).$$

Now we are in the position to give the following theorem.

\begin{theorem}\label{th4.3}
If $v\geq p+1$, then $u_t(x)$ is a solution in $\mathcal{S}_{p}^+( \mathbb{R}) $ of SPDE (\ref{eq4.10}) with the initial condition $u(0,x)=\Sigma_0\in \mathcal{S}_{p}^+( \mathbb{R}) $ for any $t\in [0,T]$ $(T\in \mathbb{R})$. Furthermore, if we assume that the solution of (\ref{eq4.10}) is a semimartingale for each $x\in \mathbb{R}$ and continuous in $(t,x)$ and $3$-times continuously differentiable in $x$ a.s., then $u_t(x)$ is the unique solution in $\mathcal{S}_{p}^+( \mathbb{R}) $ of (\ref{eq4.10}).
\end{theorem}
\begin{proof}
By Lemmas \ref{ito} and \ref{lemma4.5}, we have
\begin{equation}\label{eq4.29}
d{{\bar \eta }_t}(\xi _t^{ - 1}(x)) = vf(x){I_p}dt + g(x)\left(\sqrt {{{\bar \eta }_t}(x + t)} d{W_t} + d{W_t}'\sqrt {{{\bar \eta }_t}(x + t)} \right)
 + \partial {{\bar \eta }_t}(x + t)dt,
\end{equation}
as $d\xi _t^{ - 1}(x)=dt$ and $\langle\xi ^{ - 1}(x) \rangle_t=0$.

Since $u_t(x)={{\bar \eta }_t}(x + t)$ and
$$\partial {{\bar \eta }_t}(x + t)=\partial_y{{\bar \eta }_t}(y)|_{y=x+t}=\partial_x {{\bar \eta }_t}(x + t)=\partial_xu_t(x),$$
it is clear, by (\ref{eq4.29}), that $u_t(x)$ satisfies SPDE (\ref{eq4.10}). Furthermore, Lemma \ref{lemma4.5} shows that $u_t(x)$ is a semimartingale for each $x\in \mathbb{R}$ and continuous in $(t,x)$ and $3$-times continuously differentiable in $x$ a.s.

Assume that $\hat{u}_t(x)$ is another solution of (\ref{eq4.10}), which is a semimartingale for each $x\in \mathbb{R}$ and continuous in $(t,x)$ and $3$-times continuously differentiable in $x$ a.s. Now $\partial_x\hat{u}_t(x)$ is continuous in $(t,x)$ and twice continuously differentiable in $x$ a.s. such that all the conditions of Lemma \ref{ito} are satisfied. Thus, by Lemma \ref{ito}, we get
\begin{eqnarray*}
d{\hat{u}_t}(x - t) &=& (d{\hat{u}_t})(x - t) + {\partial _x}{\hat{u}_t}(x - t)d(x - t)\\
 &=& vf(x - t){I_p}dt + g(x - t)\left(\sqrt {{\hat{u}_t}(x - t)} d{W_t} + d{W_t}'\sqrt {{\hat{u}_t}(x - t)}\right)
\end{eqnarray*}
with $\hat{u}_0(x - 0)=\Sigma_0$, which implies that ${\hat{u}_t}(x - t)$ is the strong solution of (\ref{eq4.12}). By the uniqueness of the solution of (\ref{eq4.12}), we have ${\hat{u}_t}(x - t)=\eta_t(x,\Sigma_0)={{\bar \eta }_t}(x)$, which shows that ${\hat{u}_t}(x )={{\bar \eta }_t}(x+t)=u_t(x)$. Consequently, we have proved the theorem.\hfill $\blacksquare$
\end{proof}

\begin{definition}\label{def4.2}
For any given real number $v>0$, an $\varepsilon$-fractional Wishart process $\varepsilon$-$fWIS(H,v,p,\Sigma_0)$ with $\Sigma_0\in S_p^+(\mathbb{R})$ is defined by the solution $u_t(\varepsilon)$ of (\ref{eq4.10}) in $\mathcal{S}_p^+(\mathbb{R})$.
\end{definition}

With Theorem \ref{th4.3}, we can make sure that the $\varepsilon$-fractional Wishart processes  $\varepsilon$-$fWIS(H,v,p,\Sigma_0)$ always exist under some conditions.

\begin{remark}
As Cauchy problem (\ref{eq4.10}) has a unique solution, we know that SPDE (\ref{eq4.4}) also has a unique solution  and so the solution of (\ref{eq4.4}) can be determined by its initial condition without the boundary conditions.  If $H=\frac{1}{2}$, then SPDE (\ref{eq4.4}) degenerates to SDE considered in \cite{Bru} which governs the Wishart process. Therefore,  when $H=\frac{1}{2}$, the $\varepsilon$-fractional Wishart process defined by Definition \ref{def4.2}  reduces to the Wishart process defined in \cite{Bru} or \cite{Pfaffel}.
\end{remark}

\begin{theorem}\label{th4.4}
For $v\geq p+1$, the Laplace transform of $u_t(\varepsilon)$ in (\ref{eq4.10}) is given by
\begin{eqnarray}\label{eq4.33}
&&E[\mathrm{etr}( - Zu_t(\varepsilon ) )]\nonumber\\
&=& \det {\left({I_p} + 2({(t + \varepsilon )^{2H}} - {\varepsilon ^{2H}})Z\right)^{ - \frac{v}{2}}}\mathrm{etr}\left( - Z{({I_p} + 2({(t + \varepsilon )^{2H}} - {\varepsilon ^{2H}})Z)^{ - 1}}\Sigma _0 \right)
\end{eqnarray}
for any $Z \in \mathcal{S}^+_{p}( \mathbb{R})$.
\end{theorem}

\begin{proof}
By Theorem \ref{th4.2}, if $v$ is an integer and $v=n$ ($n\geq p+1$) ,  then the Laplace transform of $u_t(\varepsilon)$ in (\ref{eq4.10}) is given by (\ref{eq4.3}).
Let
$$
\mathcal{F}_s=\sigma\{W_{l}: 0\leq l\leq s\},\quad {\Psi _{s,h}}(Z) = E[\mathrm{etr}(-Z{\bar \eta _{s + h}}(y))|\mathcal{F}_s].
$$
By a similar method in \cite{Gourieroux06}, for any $s$, $y \in \mathbb{R}$ and $Z\in \mathcal{S}^+_p(\mathbb{R})$, we can calculate ${\Psi _{s,h}}(Z)$ as follows:
\begin{equation}\label{eq4.34}
{\Psi _{s,h}}(Z) = \exp (b(h,s,y))\mathrm{etr}(B(h,s,y){{\bar \eta }_s}(y)),
\end{equation}
where $b(h,s,y)\in \mathbb{R}$ and $B(h,s,y)\in S_p(\mathbb{R})$ satisfy the following system of PDEs:
\begin{equation}\label{eq4.35}
\left\{ {\begin{array}{*{20}{l}}
\partial_h b(h,s,y) - \partial_s b(h,s,y) = n\mathrm{Tr}[f(y - s)B(h,s,y)],\\
\partial_h B(h,s,y) - \partial_s B(h,s,y) = 2{g^2}(y - s){B^2}(h,s,y)
\end{array}} \right.
\end{equation}
with the boundary conditions $b(0,s,y)=0$ and $B(0,s,y)=-Z$.

Indeed, on the one hand, we have
\begin{equation}\label{eq4.36}
{\Psi _{s,h + ds}}(Z) = \exp \{\mathrm{Tr}[B(h + ds,s,y){\overline \eta  _s}(y)] + b(h + ds,s,y)\}.
\end{equation}
On the other hand, by law of iterated expectation, we get
\begin{eqnarray}\label{eq4.37}
&&{\Psi _{s,h + ds}}(Z)\nonumber\\
& =& E[\exp \{  - \mathrm{Tr}[Z{\overline \eta  _{s + h + ds}}(y)]\} |{\mathcal{F}_s}]\nonumber\\
& =& E[{\Psi _{s + ds,h}}(Z)|{\mathcal{F}_s}]\nonumber\\
& =& E[\exp \{ \mathrm{Tr}[B(h,s + ds,y){\overline \eta  _{s + ds}}(y)] + b(h,s + ds,y)\} |{\mathcal{F}_s}]\nonumber\\
 &=& \exp \{ b(h,s + ds,y) + \mathrm{Tr}[B(h,s + ds,y){\overline \eta  _s}(y)\} E[\mathrm{etr}\{ B(h,s + ds,y)d{\overline \eta  _s}(y)\} |{\mathcal{F}_s}]
\end{eqnarray}
and
\begin{eqnarray}\label{eq4.38}
&&E[\mathrm{etr}\{ B(h,s + ds,y)d{\overline \eta  _s}(y)\} |{\mathcal{F}_s}]\nonumber\\
 &=& \mathrm{etr}\{ nf(y - s)B(h,s + ds,y)ds\} E\left[\mathrm{etr}\{ 2g(y - s)B(h,s + ds,y)\sqrt {{{\overline \eta  }_s}(y)} d{W_s}\} |{\mathcal{F}_s}\right]\nonumber\\
 &=& \mathrm{etr}\{ nf(y - s)B(h,s + ds,y)ds + 2{g^2}(y - s){B^2}(h,s + ds,y){\overline \eta  _s}(y)ds\}.
\end{eqnarray}
By using (\ref{eq4.36})-(\ref{eq4.38}), we have
\begin{equation}\label{eq4.39}
\left\{
\begin{array}{c}
\frac{{b(h + ds,s,y) - b(h,s + ds,y)}}{{ds}} = \mathrm{Tr}[nf(y - s)B(h,s + ds,y)],\\
\frac{{B(h + ds,s,y) - B(h,s + ds,y)}}{{ds}} = 2{g^2}(y - s){B^2}(h,s + ds,y).
\end{array}
\right.
\end{equation}
Letting $ds\rightarrow 0$ in (\ref{eq4.39}), we obtain (\ref{eq4.35}).

Notice that (\ref{eq4.35}) has a unique solution because (\ref{eq4.12}) has a unique solution such that $\bar \eta _{t}(x)$ has a unique conditional Laplace transform. For any $s\geq 0$ and $h\geq 0$, we have
\begin{eqnarray}\label{eqx}
&&E[\mathrm{etr}( - Z{u_{s + h}}(\varepsilon ))|{\mathcal{F}_s}]\nonumber\\
 &=& E[\mathrm{etr}( - Z{{\bar \eta }_{s + h}}(s + h + \varepsilon ))|{\mathcal{F}_s}]\nonumber\\
 &=&{\Psi _{s,h}}(Z){|_{y = s + h + \varepsilon }}\nonumber\\
 &=& \exp (b(h,s,s + h + \varepsilon ))\mathrm{etr}(B(h,s,s + h + \varepsilon ){{\bar \eta }_s}(s + h + \varepsilon ))\nonumber\\
 &=& \exp (b(h,s,s + h + \varepsilon ))\mathrm{etr}(B(h,s,s + h + \varepsilon ){u_s}(h + \varepsilon )).
\end{eqnarray}
Let $s=0$ and $h=t$ in (\ref{eqx}), we get
\begin{equation}\label{eq4.40}
E[\mathrm{etr}( - Z{u_t}(\varepsilon))] = \exp (b(t,0,\varepsilon + t))\mathrm{etr}(B(t,0,\varepsilon + t){\Sigma _0}).
\end{equation}
Comparing (\ref{eq4.40}) with (\ref{eq4.3}), we have
\begin{equation}\label{eq4.41}
\left\{ {\begin{array}{*{20}{l}}
{b(t,0,\varepsilon + t) =  - \frac{n}{2}\ln [\det ({I_p} + 2({{(t + \varepsilon )}^{2H}} - {\varepsilon ^{2H}})Z)],}\\
{B(t,0,\varepsilon + t) =  - Z{{\left({I_p} + 2({{(t + \varepsilon )}^{2H}} - {\varepsilon ^{2H}})Z\right)}^{ - 1}}.}
\end{array}} \right.
\end{equation}

Now let $v\geq p+1$ be any real number. In a similar way, for any $s$, $y \in \mathbb{R}$ and $Z\in \mathcal{S}^+_p(\mathbb{R})$, we can obtain that the conditional Laplace transform of  ${\bar \eta _{s + h}}(y)$
$$
{\Psi^* _{s,h}}(Z) = E[\mathrm{etr}(-Z{\bar \eta _{s + h}}(y))|\mathcal{F}_s].
$$
is given by
\begin{equation}\label{eq4.4+}
{\Psi _{s,h}^*}(Z) = \exp (b^*(h,s,y))\mathrm{etr}(B^*(h,s,y){{\bar \eta }_s}(y))
\end{equation}
and we have that
\begin{equation}\label{eq4.4+4}
E(\mathrm{etr}[ - Z{u_{s + h}}(\varepsilon )]|{\mathcal{F}_s})=\exp (b^*(h,s,s + h + \varepsilon ))\mathrm{etr}[B^*(h,s,s + h + \varepsilon ){u_s}(h + \varepsilon )],
\end{equation}
where $b^*(h,s,y)\in \mathbb{R}$, $Z\in \mathcal{S}^+_p(\mathbb{R})$ and $B^*(h,s,y)\in S_p(\mathbb{R})$ satisfy the following system of PDEs:
\begin{equation}\label{eq4.4++}
\left\{ {\begin{array}{*{20}{c}}
\partial_h b^*(h,s,y) - \partial_s b^*(h,s,y) = v\mathrm{Tr}[f(y - s)B^*(h,s,y)],\\
\partial_h B^*(h,s,y) - \partial_s B^*(h,s,y) = 2{g^2}(y - s){(B^*)^2}(h,s,y)
\end{array}} \right.
\end{equation}
with the boundary conditions $b^*(0,s,y)=0$ and $B^*(0,s,y)=-Z$. Let
\begin{equation}\label{eq4.4+3}
b^*(h,s,y)=\frac{v}{n}b(h,s,y),\quad B^*(h,s,y)=B(h,s,y).
\end{equation}
It is easy to know, by comparing (\ref{eq4.4++}) with (\ref{eq4.35}), that $b^*(h,s,y)$ and $B^*(h,s,y)$ given by (\ref{eq4.4+3}) is the solution of (\ref{eq4.4++}). Now letting $s=0$ and $h=t$ in (\ref{eq4.4+}), for any real number $v\geq p+1$, it follows from (\ref{eq4.4+}), (\ref{eq4.4+3}) and (\ref{eq4.41}) that the Laplace transform of ${u_t}(\varepsilon)$ must be (\ref{eq4.33}).   \hfill $\blacksquare$
\end{proof}

\begin{remark}
We note that, if $v\geq p+1$,  then Theorem \ref{th4.4} shows that the Laplace transform of $\varepsilon$-$fWIS(H,v,p,\Sigma_0)$ converges to the Laplace transform of $fWIS^*(H,v,p,\Sigma_0)$ when $\varepsilon$ tends to $0$, which implies that $\varepsilon$-$fWIS(H,v,p,\Sigma_0)$ converges to $fWIS^*(H,v,p,\Sigma_0)$ in distribution when $\varepsilon$ tends to $0$. Consequently, we can still regard an $\varepsilon$-fractional Wishart process as an approximation of a fractional Wishart process in distribution.
\end{remark}

\begin{remark}
Comparing the $\varepsilon$-fWIS process defined by Definition \ref{def4.2} with the Wishart process in \cite{Bru}, there is an essential difference.  The  Wishart process $u_t$ has the affine property which means $u_t$ is an affine process (\cite{Gourieroux}) such that
$$
E(\mathrm{etr}[ - Z{u_{s + h}}]|{\mathcal{F}_s}) = \exp (b^*(h))\mathrm{etr}[B^*(h){u_s}].
$$
However, when $H\neq \frac{1}{2}$, it follows from (\ref{eq4.4+4}) that the $\varepsilon$-fWIS process can hardly possess such affine property. Moreover, by (\ref{eqx}) and law of iterated expectation, it is not hard to check that
$$
E[\mathrm{etr}( - Z{u_{s + h}}(\varepsilon ))|{u_s}(h + \varepsilon )] = E\left[E[\mathrm{etr}( - Z{u_{s + h}}(\varepsilon ))|{\mathcal{F}_s}]|{u_s}{(h + \varepsilon )}\right] = E[\mathrm{etr}( - Z{u_{s + h}}(\varepsilon ))|{\mathcal{F}_s}].
$$
If we take $Z=I_p$, then by the definition of the Markov process (see, for example, Definition 17.2.1 in \cite{Cohen}), we know that for each $\varepsilon>0$, if $H\neq \frac{1}{2}$, then $u_t(\varepsilon)$ ($t\in [0,T]$) is not a Markov process, which is contrary to the case of the Wishart process.
\end{remark}

\section{Generalization: six-parameter $\varepsilon$-fWIS}\label{section5}\noindent
\setcounter{equation}{0}
Let $\Omega$, $Q$ and $K$ be matrices in $\mathcal{M}_{p,p}(\mathbb{R})$. Now we introduce the following SPDE:
\begin{eqnarray}\label{eq.5.1}
d{u_t}(x) &=& \partial {u_t}(x)dt + f(x)[\Omega\Omega' + {u_t}(x)K + K'{u_t}(x)]dt\nonumber\\
 &&\mbox{}  + g(x)\left[\sqrt {{u_t}(x)} d{W_t}Q + Q'd{W_t}'\sqrt {{u_t}(x)} \right], \quad {u_0}(x) = {\Sigma _0} > 0,
\end{eqnarray}
where $f(x)$ and $g(x)$ are two functions appeared in (\ref{eq4.10}) and $W_t$ is a Brownian matrix in $\mathcal{M}_{p,p}(\mathbb{R})$. The stochastic characteristic equation associated with (\ref{eq.5.1}) is given by
\begin{equation}\label{eq.5.2}
\left\{\begin{array}{l}
d{\xi _t} =  - dt,\; {\xi _0} = x,\; a.s.\\
d{\eta _t} = f({\xi _t})\left[\Omega\Omega' + {\eta _t}K + K'{\eta _t}\right]dt + g({\xi _t})\left[\sqrt {{\eta _t}} d{W_t}Q + Q'd{W_t}'\sqrt {{\eta _t}} \right],\; {\eta _0} = {\Sigma _0} > 0, \; a.s.
\end{array}\right.
\end{equation}

Similar to Lemma \ref{lemma4.3}, we have the following result for (\ref{eq.5.2}).

\begin{lemma}
If $\Omega\Omega'-(p+1)Q'Q\in \overline{\mathcal{S}_p^+}(\mathbb{R})$, then (\ref{eq.5.2}) has a unique adapted continuous strong solution $(\xi _t,\eta _t)_{t\in \mathbb{R}_+}$ on $\mathbb{R}\times\mathcal{S}_{p}^+(\mathbb{R})$ for each initial condition $(x,\Sigma _0)\in \mathbb{R}\times\mathcal{S}_{p}^+(\mathbb{R})$. In particular, the stopping time $T(x,\Sigma _0)=\inf\{t\geq 0:\eta _t \notin \mathcal{S}_{p}^+(\mathbb{R})\}=+\infty,a.s.$.
\end{lemma}
\begin{proof}
Let $J(t,X)=f(x-t)\left[\Omega\Omega' + XK + K'X\right]$, $G(t,X)=g(x-t)Q$, $F(t,X)=\sqrt{X}$, $ff(t,X)=F(t,X)F(t,X)'$ and $gg(t,X)=G(t,X)' G(t,X)$  for any $X\in\mathcal{S}_{p}^+(\mathbb{R})$. Then we have $ff(t,X)=X$ and $gg(t,X)=f(x-t)Q'Q$ for any $X\in\mathcal{S}_{p}^+(\mathbb{R})$.

As we have showed in Lemma \ref{lemma4.3}, $G'\otimes F(t,Y)$ is locally Lipschitz and of linear growth. It is not hard to show that $J(t,X)=f(x-t)\left[\Omega\Omega' + XK + K'X\right]$ is locally Lipschitz and of linear growth.

Now, by $\Omega\Omega'-(p+1)Q'Q\in \overline{\mathcal{S}_p^+}(\mathbb{R})$, we have
\begin{align}
&\mathrm{Tr}[J(t,X) X^{-1}]-\mathrm{Tr}[ff(t,X) X^{-1}]\mathrm{Tr}[ gg(t,X) X^{-1}]\nonumber\\
&\mbox{}-\mathrm{Tr}[ff(t,X) X^{-1}gg(t,X) X^{-1}]\nonumber\\
=&\mathrm{Tr}[f(t-x)[\Omega\Omega'-(p+1)Q'Q]X^{-1}]+2\mathrm{Tr}[f(x-t)K]\nonumber\\
\geq &2\mathrm{Tr}[f(x-t)K].
\end{align}
Let $c(t)=2\mathrm{Tr}[f(x-t)K]$. Then by (\ref{eq4.15}), we have, for each $a\in \mathbb{R}_+$,
$$\int_0^a {|c(s)|ds}   <  + \infty. $$
Making use of Lemma \ref{lemma4.3}, the result is obvious. \hfill $\blacksquare$
\end{proof}

\begin{theorem}\label{th5.1}
Assume that $T$ is any positive number. If $\Omega\Omega'-(p+1)Q'Q\in \overline{\mathcal{S}_p^+}(\mathbb{R})$, then SPDE (\ref{eq.5.1}) with the initial condition $u(0,x)=\Sigma_0\in \mathcal{S}_{p}^+( \mathbb{R})$ a.s. has a solution in $\mathcal{S}_{p}^+( \mathbb{R})$ on $t\in [0,T]$. Furthermore, if we assume the solution of SPDE (\ref{eq.5.1}) is a semimartingale for each $x\in \mathbb{R}$ and continuous in $(t,x)$ and $3$-times continuously differentiable in $x$ a.s., then the solution exists uniquely.
\end{theorem}
\begin{proof}
The proof of theorem is similar to the ones of Lemma \ref{lemma4.5} and Theorem \ref{th4.3} and so we omit it here.
\hfill $\blacksquare$
\end{proof}

Now we can define the six-parameter $\varepsilon$-fractional Wishart process.
\begin{definition}\label{6w}
Assume that $\Omega,Q,K$ are matrices in $\mathcal{M}_{p,p}(\mathbb{R})$. An $\varepsilon$-fractional Wishart process with Hurst index $H$, dimension $p$, initial state $\Sigma_0\in \mathcal{S}_p^+(\mathbb{R})$ a.s. and parameters $\Omega,Q,K$, written as $\varepsilon$-$fWIS(H,p,\Sigma_0,\Omega,Q,K)$,  is defined by the solution $u_t(\varepsilon)$ of (\ref{eq.5.1}) in $\mathcal{S}_p^+(\mathbb{R})$.
\end{definition}

Obviously, $\Omega\Omega'=vQ'Q$ is a special case of Definition \ref{6w} and the parameter $v$ is kept. So we denote such case by $\varepsilon$-$fWIS(H,v,p,\Sigma_0,Q,K)$. Moreover, if $v\geq p+1$, then the condition in Theorem \ref{th5.1} holds. In this case, it is seen that if $H=\frac{1}{2}$, then (\ref{eq.5.1}) with $x\in[\varepsilon,1]$ degenerates to the SDE which governs the five-parameter Wishart process in \cite{Bru}. Hence, $\varepsilon$-$fWIS(\frac{1}{2},v,p,\Sigma_0,Q,K)$ coincides with the five-parameter Wishart process in sense of \cite{Bru}.

Next, we roughly discuss serial correlation properties of the $\varepsilon$-fractional Wishart process.

Let $t>0$ and $\Delta t>0$ be two real numbers and define the following maps:
\begin{eqnarray*}
F(s,x,{{\bar \eta }_s}(x))& = &f(x - s)[\Omega\Omega' + {{\bar \eta }_s}(x)K + K'{{\bar \eta }_s}(x)],\\
G(s,x,{{\bar \eta }_s}(x)) &=& g(x - s)\sqrt {{{\bar \eta }_s}(x)},
\end{eqnarray*}
where ${{\bar \eta }_s}(x)=\eta _s(x,\Sigma_0)$ is a matrix appeared in (\ref{eq.5.2}). Then $u_s(x)={{\bar \eta }_s}(s+x)$ and  the increment of $\varepsilon$-$fWIS(H,p,\Sigma_0,\Omega,Q,K)$ can be calculated as follows:
\begin{eqnarray}\label{eq.5.6}
\Delta {u_t}(\varepsilon ) &=& {u_{t + \Delta t}}(\varepsilon ) - {u_t}(\varepsilon )\nonumber\\
 &=& {{\bar \eta }_{t + \Delta t}}(t + \Delta t + \varepsilon ) - {{\bar \eta }_t}(t + \varepsilon )\nonumber\\
 &=& \int_0^t {[F(s,t + \Delta t + \varepsilon ,{{\bar \eta }_s}(t + \Delta t + \varepsilon )) - F(s,t + \varepsilon ,{{\bar \eta }_s}(t + \varepsilon ))]ds} \nonumber\\
 &&\mbox{}+ \int_0^t {[G(s,t + \Delta t + \varepsilon ,{{\bar \eta }_s}(t + \Delta t + \varepsilon )) - G(s,t + \varepsilon ,{{\bar \eta }_s}(t + \varepsilon ))]d{W_s}Q} \nonumber\\
 &&\mbox{}+ \int_0^t {Q'd{W_s}'[G(s,t + \Delta t + \varepsilon ,{{\bar \eta }_s}(t + \Delta t + \varepsilon )) - G(s,t + \varepsilon ,{{\bar \eta }_s}(t + \varepsilon ))]} \nonumber\\
 &&\mbox{}+ \int_t^{t + \Delta t} {F(s,t + \Delta t + \varepsilon ,{{\bar \eta }_s}(t + \Delta t + \varepsilon ))ds} \nonumber\\
 &&\mbox{}+ \int_t^{t + \Delta t} {G(s,t + \Delta t + \varepsilon ,{{\bar \eta }_s}(t + \Delta t + \varepsilon ))d{W_s}Q} \nonumber\\
 &&\mbox{}+ \int_t^{t + \Delta t} {Q'd{W_s}'G(s,t + \Delta t + \varepsilon ,{{\bar \eta }_s}(t + \Delta t + \varepsilon ))}.
\end{eqnarray}
In the case $H\neq \frac{1}{2}$, it is easy to see that
$$G(s,t + \Delta t + \varepsilon ,{{\bar \eta }_s}(t + \Delta t + \varepsilon )) - G(s,t + \varepsilon ,{{\bar \eta }_s}(t + \varepsilon ))\neq 0$$
and so the increment $\Delta {u_t}(\varepsilon )$ depends on the past information $W_s\ (0\leq s \leq t)$, which depicts the serial correlation of $\varepsilon$-$fWIS$ process, while in the case $H=\frac{1}{2}$, we have
$$G(s,t + \Delta t + \varepsilon ,{{\bar \eta }_s}(t + \Delta t + \varepsilon )) - G(s,t + \varepsilon ,{{\bar \eta }_s}(t + \varepsilon ))= 0$$
and
$$F(s,t + \Delta t + \varepsilon ,{{\bar \eta }_s}(t + \Delta t + \varepsilon )) - F(s,t + \varepsilon ,{{\bar \eta }_s}(t + \varepsilon ))= 0$$
for any $s\in[0,t]$ (because in the case $H=\frac{1}{2}$, ${{u}_s}(x)$ ($x\in[\varepsilon,a]$) does not depend on $x$ and one can choose $a$ in Remark \ref{remark4+} such that both $t + \Delta t + \varepsilon-s$ and $t + \varepsilon-s$ in the interval $[\varepsilon,a]$ for any $s\in[0,t]$), which implies the increment $\Delta {u_t}(\varepsilon )$ has no memory of the past.

\section{Applications to stochastic volatility}\label{section6}\noindent
\setcounter{equation}{0}
In this section, we shall apply the $\varepsilon$-fractional Wishart process to the stochastic volatility in a finite horizon $T<+\infty$. Let $(\mathbf{\Omega},\mathcal{F},(\mathcal{F}_t)_{t\geq 0}, \mathbb{P})$ be a filtered probability space satisfying the usual conditions.

Similar to the Wishart Affine Stochastic Correlation (WASC) model (\cite{Fonseca}), we assume that a $p$-dimensional risky asset $S_t$ whose dynamics are given by
\begin{eqnarray}\label{model}
\left\{
\begin{array}{rll}
d{S_t} &=& S_t^*\left(\mu dt + \sqrt {{u_t}(x)} d{B_t}\right);\\
d{u_t}(x) &=& \partial {u_t}(x)dt + f(x)\left[\Omega\Omega' + {u_t}(x)K + K'{u_t}(x)\right]dt\\
\mbox{} & \mbox{} & \mbox{}+ g(x)\left[\sqrt {{u_t}(x)} d{W_t}Q + Q'd{W_t}'\sqrt {{u_t}(x)} \right],
\end{array}
\right.
\end{eqnarray}
where $\mu \in \mathcal{M}_{p,1}(\mathbb{R})$ is the appreciation rate of the asset and $\Omega$, $Q$, $K$ are matrices in $\mathcal{M}_{p,p}(\mathbb{R})$ with $Q$ invertible and $\Omega\Omega'-(p+1)Q'Q\in \overline{\mathcal{S}_p^+}(\mathbb{R})$; and $S^*$ is the following matrix in $\mathcal{M}_{p,p}(\mathbb{R})$:
\[\left( {\begin{array}{*{20}{c}}
{{S_1}}&0& \cdots &0\\
0&{{S_2}}& \cdots &0\\
 \vdots & \vdots & \ddots & \vdots \\
0&0& \cdots &{{S_p}}
\end{array}} \right).\]
Here we shall assume that the Brownian matrix $B_t\in \mathcal{M}_{p,1}(\mathbb{R})$ and $W_t\in \mathcal{M}_{p,p}(\mathbb{R})$ are correlated by ${B_t} = {W_t} \rho  + \sqrt {1 - {\rho'} \rho } {H_t}$, where $H_t$ is another $p$-dimensional Brownian motion independent of $W_t$ and $\rho=(\rho_1,\cdots,\rho_p)'$ is a fixed correlation vector between the returns and the state variables with $\rho_1,\cdots,\rho_p\in [-1,1]$. Furthermore, assuming that $u_t(x)$ is a semimartingale for each $x\in \mathbb{R}$ and continuous in $(t,x)$ and $3$-times continuously differentiable in $x$ a.s., we know that $u_t(x)$ exists uniquely by Theorem \ref{th5.1}.

Now define the return process of the asset $Y_t=(\ln S_1(t),\cdots,\ln S_p(t))'$ and by It\^{o}'s lemma, we have
\[d{Y_t} = \left(\mu  - \frac{1}{2}\mathrm{diag}({u_t}(x))\right)dt + \sqrt {{u_t}(x)} d{B_t},\]
where $\mathrm{diag}({u_t}(x))=({u_{11,t}}(x),\cdots,{u_{pp,t}}(x))'$.

Set $x=\varepsilon$. Let $E_t$ denote the expectation under the condition $\mathcal{F}_t$ and similar meaning be given for $\mathrm{Var}_t$ and $\mathrm{Corr}_t$. By directly computing, we have
\[E_t[d{Y_t}]=E[d{Y_t}|{\mathcal{F}_t}] = \left(\mu  - \frac{1}{2}\mathrm{diag}({u_t}(\varepsilon))\right)dt\]
and
\[E_t[d{u_t(\varepsilon)}]=E[d{u_t(\varepsilon)}|{\mathcal{F}_t}] = \partial {u_t}(\varepsilon)dt + f(\varepsilon)[\Omega\Omega' + {u_t}(\varepsilon)K + K'{u_t}(\varepsilon)]dt.\]
Same as the discussion in \cite{Romo15}, we can calculate the conditional covariance matrix of the return process as
\begin{eqnarray}\label{eq6.0.1}
{V_t} &=& \frac{1}{{dt}}\mathrm{Var}_t(d{Y_t}) = \frac{1}{{dt}}E[(d{Y_t} - Ed{Y_t})(d{Y_t} - Ed{Y_t})'|{\mathcal{F}_t}]\nonumber\\
 &=& \frac{1}{{dt}}E\left[\sqrt {{u_t}(\varepsilon )} d{B_t}d{B_t}'\sqrt {{u_t}(\varepsilon )} |{\mathcal{F}_t}\right]\nonumber\\
 &=& {u_t}(\varepsilon )
\end{eqnarray}
and the correlation between the asset $i$ and asset $j$ ($i\neq j$) as
\begin{equation}\label{eq6.0.2}
\mathrm{Corr}_t(dY_{it},dY_{jt})=\frac{u_{ij,t}(\varepsilon )}{\sqrt {u_{ii,t}(\varepsilon )u_{jj,t}(\varepsilon )}}.
\end{equation}

Similarly, we can calculate the correlation between each asset return and its own instantaneous variance $\mathrm{Corr}_t(dY_{it},dV_{ii,t})$ and the correlation between each variance process $\mathrm{Corr}_t(dV_{ii,t},dV_{jj,t})$ as follows:
\begin{equation}\label{eq6.0.3}
\mathrm{Corr}{_t}(d{Y_{it}},d{V_{ii,t}}) = \frac{{{{(Q'\rho )}_i}}}{{\sqrt {{{(Q'Q)}_{ii}}} }},
\end{equation}
\begin{equation}\label{eq6.0.4}
\mathrm{Corr}{_t}(d{V_{ii,t}},d{V_{jj,t}}) = \frac{{{{(Q'Q)}_{ij}}{u_{ij,t}}(\varepsilon )}}{{\sqrt {{{(Q'Q)}_{ii}}{{(Q'Q)}_{jj}}{u_{ii,t}}(\varepsilon ){u_{jj,t}}(\varepsilon )} }}.
\end{equation}

From (\ref{eq6.0.1})-(\ref{eq6.0.4}), it is easy to see that, same as WASC model, $\varepsilon$-fraction Wishart volatility model (\ref{model}) also allows us to introduce stochastic
volatilities of the assets and stochastic correlations not only between the underlying assets¡¯ returns but also between their volatilities. Model (\ref{model}) still keeps the phenomena in Heston's model and WASC model considered in \cite{Romo15}, that is,  the correlation between each asset return and its own instantaneous variance is constant and there exists an implied volatility skew through $\rho$. More importantly, in the case $H\neq \frac{1}{2}$, by (\ref{eq.5.6}), the conditional covariance matrix of the return process given by (\ref{eq6.0.1}) captures the stochastic serial correlation which cannot be depicted by WSAC model.

\begin{remark}\label{heston}
If $n=1$, $H=\frac{1}{2}$, $K<0$ and $x\in[\varepsilon,1]$ in (\ref{model}), then $u_t(x)$ shall degenerate to the following SDE:
\begin{equation}\label{eqremark6.1}
du_t=-2K\left(\frac{vQ^2}{-2K}-u_t\right)dt+2Q\sqrt {u_t}dW_t,
\end{equation}
which is just the Heston's model considered in \cite{Heston}.
\end{remark}

\begin{remark}
Serial correlation among the increments of fWIS and $\varepsilon$-fWIS processes in the case $H\neq \frac{1}{2}$ makes it harder to use risk-neutral pricing method. Taking the Heston's model (see Appendix in \cite{Heston}) as an example, let $S(t)$ and $u(t)$ be the Heston's model in Remark \ref{heston}, and $\vartheta$ is the terminal payoff function at the delivery time. Since with the Markov property we have
\begin{equation}\label{heston2}
E[\vartheta(S(T),u(T))|\mathcal{F}_t]=E[\vartheta(S(T),u(T))|S(t),u(t)],
\end{equation}
there is an essential martingale to obtain the pricing PDE:
\begin{equation}\label{heston3}
V(S,u,t) = E[\vartheta(S(T),u(T))|S(t) = S,u(t) = u],
\end{equation}
when $H\neq\frac{1}{2}$, (\ref{heston2}) does not hold and so we can not use a similar martingale  (\ref{heston3}) to derive the pricing PDE.
\end{remark}

In order to show the influence of serial correlation, we give a simple example. In the risk-neutral probability space $(\mathbf{\Omega},\mathcal{F},(\mathcal{F}_t)_{t\geq 0}, \mathbb{\tilde{P}})$, assume that there exists a one-dimensional risky asset $S_t$ whose dynamics of the volatility is
\begin{equation}\label{forward1}
d{u_t}(\varepsilon) = ({\partial _x}{u_t}(\varepsilon) + vf(\varepsilon){I_p})dt + 2g(\varepsilon)\sqrt {{u_t}(\varepsilon)} d{W_t} ,
\end{equation}
where the functions and parameters are all same like before. In this financial model, we price a forward contract for instantaneous variance with a payoff $\vartheta({u_T}(\varepsilon))={u_T}(\varepsilon)-\iota$, where $T>0$ is the delivery time and $\iota$ is the delivery price. Such contract is also consider in \cite{Stojanovic10} for the utility-based pricing.

By the risk-neutral pricing method, we know that the value at $t=0$ of the forward contract is
\[V(0)=e^{-rT}\tilde{E}[{u_T}(\varepsilon)-\iota |\mathcal{F}_0],\]
where $r$ is the risk-free interest. Making use of (\ref{eq4.33}), we obtain
\begin{align*}
&\tilde{E}[-u_t(\varepsilon )\exp( - Zu_t(\varepsilon ) )]\nonumber\\
=&  -v({(t + \varepsilon )^{2H}} - {\varepsilon ^{2H}}){\left(1 + 2({(t + \varepsilon )^{2H}} - {\varepsilon ^{2H}})Z\right)^{ - \frac{v}{2}-1}}\exp\left( - Z{(1 + 2Z({(t + \varepsilon )^{2H}} - {\varepsilon ^{2H}}))^{ - 1}}\Sigma _0 \right)\nonumber\\
&\mbox{}+\left[4Z({(t + \varepsilon )^{2H}} - {\varepsilon ^{2H}}){(1 + 2Z({(t + \varepsilon )^{2H}} - {\varepsilon ^{2H}}))^{ - 2}}-{(1 + 2Z({(t + \varepsilon )^{2H}} - {\varepsilon ^{2H}}))^{ - 1}}\right]\Sigma _0 \nonumber\\
&\times{\left(1 + 2({(t + \varepsilon )^{2H}} - {\varepsilon ^{2H}})Z\right)^{ - \frac{v}{2}}}
\exp\left( - Z{(1 + 2Z({(t + \varepsilon )^{2H}} - {\varepsilon ^{2H}}))^{ - 1}}\Sigma _0 \right),
\end{align*}
for $Z>0$. Letting $Z\rightarrow 0$, it follows that
\[\tilde{E}[u_t(\varepsilon )]=v({(t + \varepsilon )^{2H}} - {\varepsilon ^{2H}})+\Sigma_0.\]
Consequently, we give
\begin{equation}\label{forward2}
V(0)=e^{-rT}\left[v({(T + \varepsilon )^{2H}} - {\varepsilon ^{2H}})+\Sigma_0-\iota\right].
\end{equation}
And the associated price at $t=0$ of the forward contract is
\[P(0)=v({(T + \varepsilon )^{2H}} - {\varepsilon ^{2H}})+\Sigma_0.\]

In fact, from the proof of Theorem \ref{th4.4} we know that $V(t_0)=\tilde{E}[u_T(\varepsilon )|\mathcal{F}_{t_0}]$ depends on $u_{t_0}(T-t_0+\varepsilon)$. Since $u_0(x)=\Sigma_0$ for $x\in \mathbb{R}$, (\ref{forward2}) is a special case.
 If $H=\frac{1}{2}$, then we have $u_{t_0}(T-{t_0}+\varepsilon)=u_{t_0}(\varepsilon)$ and we can use the value of the volatility at $t$ to price the contract. However, if $H\neq\frac{1}{2}$ and ${t_0}>0$, then $u_{t_0}(T-{t_0}+\varepsilon)=u_{t_0}(\varepsilon)$ does not necessarily hold and we can neither use the the value of the volatility at $t_0$ to price the contract nor get the information of $u_{t_0}(T-t_0+\varepsilon)$.

\section{Conclusions}\label{section7}\noindent
\setcounter{equation}{0}
In this paper, we firstly define the fractional Wishart process to capture the property of serial correlation. As the fractional Wishart process is difficult to describe in the  form of dynamics, we introduce the $\varepsilon$-fractional Wishart process to approximate the fractional Wishart process in distribution and extend the $\varepsilon$-fractional Wishart process to the case with a non-integer index.  Moreover, we have showed that the $\varepsilon$-fWIS process can be defined by the associated SPDE so that it is more applicable for some real dynamic problems.

However, serial correlation among the increments of fWIS and $\varepsilon$-fWIS processes leads to the difficulty in risk-neutral pricing method. We note that the fWIS or $\varepsilon$-fWIS process provides an interesting and useful tool to depict some complex financial models.  Therefore,  investments and pricing methods based on the fWIS or $\varepsilon$-fWIS process shall be our future work.

\end{document}